\documentclass[12pt]{amsart}
\usepackage{enumerate,amstext,amsfonts,amssymb,amscd,amsbsy,amsmath,verbatim}
\usepackage{ifthen,multicol,extarrows,fullpage,parskip,anyfontsize,mathtools,multirow}
\usepackage[lite,alphabetic]{amsrefs} 
\usepackage{lscape}
\usepackage{enumitem}
\usepackage{color,tikz}
\usepackage{amsthm}
\usepackage{latexsym}
\usepackage[all]{xy}
\usepackage[bookmarks, colorlinks=true, linkcolor=blue, citecolor=blue, urlcolor=blue]{hyperref}
\usepackage{tikz-cd}

\usepackage{eucal}

\numberwithin{equation}{section}
\newtheorem{lemma}[equation]{Lemma}

\newtheorem{algorithm}[equation]{Algorithm}

\newtheorem{conj}[equation]{Conjecture}

\newtheorem{claim*}{Claim}

\newtheorem{defn}[equation]{Definition}

\newtheorem{example}[equation]{Example}
\newtheorem{question}[equation]{Question}

\theoremstyle{remark}
\newtheorem{remark}[equation]{Remark}

\newcommand{\lex}{\operatorname{lex}}

\newcommand{\im}{\operatorname{im}}

\newcommand{\Tor}{\operatorname{Tor}}
\newcommand{\HS}{\operatorname{HS}}

\newcommand{\rank}{\operatorname{rank}}

\newcommand{\Sym}{\operatorname{Sym}} 
\newcommand{\GL}{\mathbf{GL}}

\newcommand{\defi}[1]{\textsf{#1}} 

\newcommand{\PP}{\mathbb P}

\newcommand{\QQ}{\mathbb Q}
\newcommand{\ZZ}{\mathbb Z}

\newcommand{\CC}{\mathbb C}
\newcommand{\ba}{\mathbf a}
\newcommand{\bb}{\mathbf b}

\newcommand{\bd}{\mathbf d}

\newcommand{\rE}{\mathrm E}

\newcommand{\bF}{\mathbf F}

\newcommand{\bS}{\mathbf{S}}

\newcommand{\cO}{\mathcal O}

\newcommand{\dw}{\operatorname{domWeights}}

\usepackage{rotating}

\title{Conjectures and computations about Veronese syzygies}
\author{Juliette Bruce}
\address{Department of Mathematics, University of Wisconsin, Madison, WI}
\email{\href{mailto:juliette.bruce@math.wisc.edu}{juliette.bruce@math.wisc.edu}}
\urladdr{\url{http://math.wisc.edu/~juliettebruce/}}

\author{Daniel Erman}
\address{Department of Mathematics, University of Wisconsin, Madison, WI}
\email{\href{mailto:derman@math.wisc.edu}{derman@math.wisc.edu}}
\urladdr{\url{http://math.wisc.edu/~derman/}}

\author{Steve Goldstein}
\address{Department of Mathematics, University of Wisconsin, Madison, WI}
\email{\href{mailto:sgoldstein@math.wisc.edu}{sgoldstein@math.wisc.edu}}

\author{Jay Yang}
\address{Department of Mathematics, University of Wisconsin, Madison, WI}
\email{\href{mailto:yangjay@math.wisc.edu}{yangjay@math.wisc.edu}}
\urladdr{\url{http://math.wisc.edu/~yangjay/}}
\date{\today}

\begin{document}
\thanks{
JB received support from the NSF GRFP under grant DGE-1256259, and from the Graduate School and the Office of the Vice Chancellor for Research and Graduate Education at the University of Wisconsin-Madison with funding from the Wisconsin Alumni Research Foundation.  DE received support from NSF grant DMS-1601619. JY received support from NSF grant DMS-1502553.
}
\begin{abstract}
We formulate several conjectures which shed light on the structure of Veronese syzygies of projective spaces. Our conjectures are based on experimental data that we derived by developing a numerical linear algebra and distributed computation technique for computing and synthesizing new cases of Veronese embeddings for $\PP^2$. 
\end{abstract}
\maketitle
 
A central open question in the study of syzygies is to determine the Betti table of $\PP^n$ under the $d$-uple Veronese embedding. While the case $n=1$ is well understood -- the resolution is an Eagon-Northcott complex --  even the case $n=2$ is wide open. In this paper, we formulate several conjectures which shed light on the structure of Veronese syzygies of projective spaces. For instance, Conjecture~\ref{conj:domweights} predicts the most dominant torus (or Schur functor) weights that will arise in each entry of the Betti table of $\PP^n$ under any $d$-uple embedding. Our conjectures are based on experimental data gathered using new techniques for computings syzygies of Veronese embeddings of $\PP^2$. These techniques are based upon the use of numerical linear algebra and distributed computation.
 
For a fixed $n$, let $S=\CC[x_0,x_1,\dots,x_n]$ be the polynomial ring with the standard grading. We are primarily interested in syzygies in the $d$th Veronese subring of $S$, which we denote $S(0;d):=S^{(d)}=\oplus_{i\in \ZZ} S_{di}$. We consider $S(0;d)$ as an $R$-module, where $R=\Sym(S_d)$ is the symmetric algebra on the vector space $S_d$. Geometrically, this corresponds to computing the syzygies of $\PP^n$ under the $d$-uple embedding $\PP^n \to \PP^{\binom{n+d}{d}-1}$. 

Since Green's landmark~\cite{green-I}, the syzygies of a variety are often studied in parallel with the syzygies of the other line bundles on the variety, as this provides a unifying perspective (see also \cite[Theorem~2.2]{green-II},~\cite[Theorem~2]{ein-lazarsfeld-np}, \cite[Theorem~4.1]{ein-lazarsfeld}). Accordingly, we set $S(b;d):=\oplus_{i\in \ZZ} S_{di+b}$ as an $R$-module; this is the graded $R$-module associated to the pushforward of $\cO_{\PP^n}(b)$ under the $d$-uple embedding. 

We analyze the Betti numbers of $S(b;d)$, as well as multigraded and equivariant refinements. We write
\begin{align*}
K_{p,q}(\PP^n, b;d)&=\Tor^{R}_{p}(S(b;d),\CC)_{p+q}=\CC^{\beta_{p,p+q}(\PP^n,b;d)}.\\
\intertext{Thus $\beta_{p,p+q}(\PP^n,b;d)$ denotes the vector space dimension of $K_{p,q}(\PP^n,b;d)$. The natural linear action of $\GL_{n+1}(\CC)$ on $S$ induces an action on $K_{p,q}(\PP^n, b;d)$, and so we can decompose this as a direct sum of Schur functors of total weight $d(p+q)+b$ i.e.}
K_{p,q}(\PP^n, b;d) &= \bigoplus_{\substack{\lambda \text{ of weight }\\ d(p+q)+b}} \bS_{\lambda}(\CC^{n+1})^{\oplus m_{\lambda}},\\
\intertext{where $\bS_{\lambda}$ is the Schur functor corresponding to the partition $\lambda$~\cite[p. 76]{fulton-harris}. This is the \defi{Schur decomposition} of $K_{p,q}(\PP^n, b;d)$, and is the most compact way to encode the syzygies.  Specializing to the action of $(\CC^*)^{n+1}$, gives a decomposition of $K_{p,q}(\PP^n, b;d)$ into a sum of $\ZZ^{n+1}$-graded vector spaces of total weight $d(p+q)+b$. Specifically, writing $\CC(-\mathbf{a})$ for the vector space $\CC$ together with the $(\CC^*)^{n+1}$-action given by $(\lambda_0,\lambda_1,\dots,\lambda_n)\cdot \mu = \lambda_0^{a_0}\lambda_1^{a_1}\cdots \lambda_n^{a_n}\mu$ we have}
K_{p,q}(\PP^n, b;d) &= \bigoplus_{\substack{\mathbf{a}\in \ZZ^{n+1} \\|\mathbf{a}|=d(p+q)+b}} \CC(-\mathbf{a})^{\oplus \beta_{p,\mathbf{a}}(\PP^n,b;d)}
\end{align*}
as a $\ZZ^{n+1}$-graded vector spaces, or equivalently as $(\CC^*)^{n+1}$ representations. This is referred to as the \defi{multigraded decomposition} of $K_{p,q}(\PP^n, b;d)$.

We are motivated by three main questions. The most ambitious goal is to provide a full description of the Betti table of every Veronese embedding in terms of Schur modules.
\begin{question}[Schur Modules]\label{q:schur}
Compute the Schur module decomposition of $K_{p,q}(\PP^n, b;d)$.
\end{question}
Almost nothing is known, or even conjectured, about this question, even in the case of $\PP^2$. Our most significant conjecture provides a first step towards an answer to this question. Specifically, Conjecture~\ref{conj:domweights} proposes an explicit prediction for the Schur modules $\bS_{\lambda}\subseteq K_{p,q}(\PP^n,b;d)$ with the most dominant weights.

Our second question comes from Ein and Lazarsfeld's~\cite[Conjecture 7.5]{ein-lazarsfeld} and is related to more classical questions about Green's $N_p$-condition for varieties~\cite{green-I, ein-lazarsfeld-np}:
\begin{question}[Vanishing]\label{q:vanishing}
When is $K_{p,q}(\PP^n, b;d)=0$?
\end{question}
Our Conjecture~\ref{conj:domweights} would also imply ~\cite[Conjecture 7.5]{ein-lazarsfeld}, and thus it offers a new perspective on Question~\ref{q:vanishing}.  Conjecture~\ref{conj:domweights} is based on a construction of monomial syzygies, introduced in~\cite{eel-quick}. Our new data suggests a surprisingly tight correspondence between the dominant weights of $K_{p,q}(\PP^n,b;d)$ and the monomial syzygies constructed in~\cite{eel-quick}, and that there is much more to be understood from this simple monomial construction.

\noindent Our third question is inspired by Ein, Erman, and Lazarsfeld's conjecture that each row of these Betti tables converges to a normal distribution~\cite[Conjecture~B]{eel-betti}.
\begin{question}[Quantitative Behavior]\label{q:normal distribution}
Fix $n,q$ and $b$. 
\begin{enumerate}[noitemsep]
	\item Can one provide any reasonable quantitative description or bounds on $K_{p,q}(\PP^n, b; d)$, either for a fixed $d$ or as $d\to \infty$?
	\item More specifically, does the function $p\mapsto \dim K_{p,q}(\PP^n, b; d)$, when appropriately scaled, converge to a normal distribution as $d\to \infty$?
\end{enumerate}
\end{question}
We provide some of the first evidence for the normally distributed behavior conjectured in~\cite[Conjecture~B]{eel-betti} -- see Figure~\ref{fig:normal distribution} and \S\ref{subsec:normal dist}. 

Additionally we produce an array of new conjectures related to Questions~\ref{q:schur} and \ref{q:normal distribution}, including conjectures on: Boij-S\"oderberg coefficients; the number of (disctinct) Schur modules appearing in $K_{p,q}(\PP^n, b;d)$; and a Schur functor interpretation of the conjecture of~\cite[\S8.3]{wcdl}.  Our conjectures are based on new experimental data about the $K_{p,q}(\PP^2, b;d)$ that arose from large-scale, systematic computations.  Taken together, these new conjectures sharpen our understanding of Veronese syzygies, and provide tangible projects to explore. 

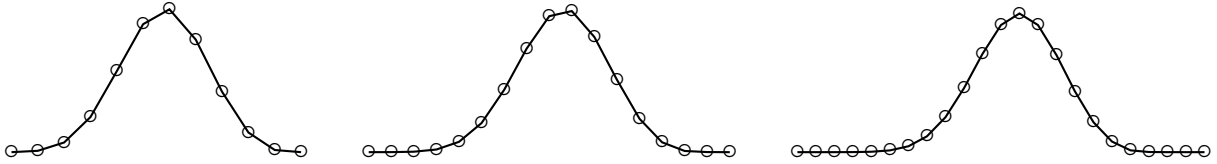
\begin{figure}
\begin{center}
\begin{tikzpicture}[yscale=2.4,xscale = .35]
\draw (0,0) node{$\circ$};
\draw (1,.0075) node{$\circ$};
\draw (2,.0536) node{$\circ$};
\draw (3,.1947) node{$\circ$};
\draw (4,.4488) node{$\circ$};
\draw (5,.7095)node{$\circ$};
\draw (6,.792) node{$\circ$};
\draw (7,.6237) node{$\circ$};
\draw (8,.3344) node{$\circ$};
\draw (9,.1089) node{$\circ$};
\draw (10,.012) node{$\circ$};
\draw (11,0) node{$\circ$};
\draw[-,thick] (0,0)--(1,.0075)--(2,.0536)--(3,.1947)--(4,.4488)--(5,.7095)--(6,.792)--(7,.6237)--(8,.3344)--(9,.1089)--(10,.012)--(11,0);
\end{tikzpicture}\quad \begin{tikzpicture}[yscale=.033,xscale = .30]
\draw (0,0) node{$\circ$};
\draw (1,.0165) node{$\circ$};
\draw (2,.1830) node{$\circ$};
\draw (3,1.0710) node{$\circ$};
\draw (4,4.1616) node{$\circ$};
\draw (5,11.7300) node{$\circ$};
\draw (6,25.0920)node{$\circ$};
\draw (7,41.7690)node{$\circ$};
\draw (8,54.8080) node{$\circ$};
\draw (9,56.8854) node{$\circ$};
\draw (10,46.4100) node{$\circ$};
\draw (11,29.1720) node{$\circ$};
\draw (12,13.4640) node{$\circ$};
\draw (13,3.9780) node{$\circ$};
\draw (14,.4855) node{$\circ$};
\draw (15,.0375) node{$\circ$};
\draw (16,0) node{$\circ$};
\draw[-,thick] (0,0)--(1,.0165)--(2,.1830)--(3,1.0710)--(4,4.1616)--(5,11.7300)--(6,25.0920)--(7,41.7690)--
(8,54.8080)--(9,56.8854)--(10,46.4100)--(11,29.1720)--(12,13.4640)--(13,3.9780)--(14,.4855)--(15,.0375)--(16,0);
\end{tikzpicture}\quad \begin{tikzpicture}[yscale=.023,xscale = .246]
\draw (0,0) node{$\circ$};
\draw (1,.000315) node{$\circ$};
\draw (2,.00495) node{$\circ$};
\draw (3,.04185) node{$\circ$};
\draw (4,.24012) node{$\circ$};
\draw (5,1.02465) node{$\circ$};
\draw (6,3.4155)node{$\circ$};
\draw (7,9.16493)node{$\circ$};
\draw (8,20.1894) node{$\circ$};
\draw (9,36.9899)node{$\circ$};
\draw (10,56.8319) node{$\circ$};
\draw (11,73.5471) node{$\circ$};
\draw (12,80.2332) node{$\circ$};
\draw (13,73.5471) node{$\circ$};
\draw (14,56.1632) node{$\circ$};
\draw (15,35.102) node{$\circ$};
\draw (16,17.3052) node{$\circ$};
\draw (17,6.1303) node{$\circ$};
\draw (18,1.04993) node{$\circ$};
\draw (19,.091855) node{$\circ$};
\draw (20,.01141) node{$\circ$};
\draw (21,.000945) node{$\circ$};
\draw (22,0) node{$\circ$};
\draw[-,thick] (0,0)--(1,.000315)--(2,.00495)--(3,.04185)--(4,.24012)--(5,1.02465)--(6,3.4155)--(7,9.16493)-- (8,20.1894)--(9,36.9899)--(10,56.8319)--(11,73.5471)--(12,80.2332)--(13,73.5471)--(14,56.1632)--(15,35.102)--(16,17.3052)--(17,6.1303)--(18,1.04993)--(19,.091855)--(20,.01141)--(21,.000945)--(22,0);
\end{tikzpicture}
\end{center}
\caption{Plots of $p\mapsto \dim K_{p,1}(0;d)$ for $d=4,5,$ and $6$ suggest that the Betti numbers of the quadratic strand of the Veronese embeddings of $\PP^2$ convergence towards a normal distribution as $d\to \infty$.}
\label{fig:normal distribution}
\end{figure}

We computed the $K_{p,q}(\PP^2,b;d)$ spaces for all $p,q$ and essentially\footnote{Due to the ongoing nature of this experiment, the end goal is a moving target. As of the writing of this paper, a few multigraded entries for $d=6$ and $b=4,5$ were still running, and a few Schur functor entries for $d=6$ and $b=0,3$ were unprocessed. On the other hand, we also produced some data for $d=7,8$.} all $0\leq b\leq d\leq 6$, as well as the corresponding Schur module decompositions and multigraded Hilbert series. For comparison: our Macaulay2 computations did not terminate for $d=5$ and $b=0$; this case, including multigraded decompositions, was recently computed by \cite{greco-vero-d5};  and the case $d=6$ and $b=0$, including multigraded decompositions, was computed even more recently \cite{wcdl}.  The main contribution of our experimental data is thus its comprehensiveness, as we include the pushforwards of other line bundles, the Schur functor decompositions, and more.

\begin{figure}
\[
\begin{array}{c|c|c|c}
&&\# \text{ of Relevant} & \text{Largest} \\ 
d&b& \text{Matrices}& \text{Matrix}\\\hline
\multirow{6}{*}{6}&0& 1,028 &596,898\times 1,246,254\\ \cline{2-4}
&1&148 &7,345 \times 9,890\\ \cline{2-4}
&2&148 &7,345 \times 9,890\\ \cline{2-4}
&3&1,028 &596,898\times 1,246,254\\ \cline{2-4}
&4& 1,753 & 4,175,947\times 12,168,528\\ \cline{2-4}
&5&1,753 &4,175,947\times 12,168,528\\ \hline
\end{array}
\]
\caption{This table summarizes data about the matrices involved in our computations of the Veronese syzygies of $\PP^2$ when $d=6$. We include it here to give a hint of the scale of computation involved. See \S\ref{sec:main computation} for more details.}
\label{fig:d6}
\end{figure}

Our computation is not based on new mathematical ideas, but rather in the synthesis of known results and the coordinated execution of many elementary steps. Since Betti numbers are Tor groups, they can be computed in  two ways. The standard method is to use symbolic algebra algorithms to compute a minimal free resolution, and to derive the Betti table from this resolution~\cite[Chapter 2]{M2-book}. This method is quite computationally intensive, and does not terminate for $d\geq 5$. 

A second method is to compute the cohomology of the Koszul complex, which reduces the computation of these Tor groups to linear algebra (see \S\ref{sec:computational approach} below). Despite this reduction long being know we are aware of only one large-scale effort at using it to compute Betti numbers  \cite{wcdl}. This is likely because, even for relatively simple cases, the problem remains quite complicated; the matrices quickly become massive and numerous. 

The crux of our technique is the use of high-speed high-throughput computing to compute multigraded Betti numbers. This allows us to compute each multigraded Betti number in parellel, relying on numerical linear algebra algorithms, in particular an LU-decomposition algorithm~\cite{LUSOL}. These algorithms are numerical in nature, and so rounding errors may creep in. However, our primary interest is in the testing and development of conjectures, so we do not require the precision of symbolic computation. Moreover, as discussed in \S\ref{sec:post processing}, we can often correct for minor errors through a post-processing step, which converts the multigraded decomposition into the Schur functor decomposition.

We have made our experimental data public in several formats. This includes a public database:
\href{https://syzygydata.com}{syzygydata.com} where the results of all computations have been presented and organized. It also includes a Macaulay2 package (in preparation) that incorporates the output of all computations. Our goal is to make our data readily accessible to others in hope of spurring further work on Veronese syzygies. 

This paper is organized as follows. \S\ref{sec:background} provides background and notation. \S\ref{sec:computational approach} gives an outline of our computation. This is elaborated upon in \S\ref{sec:precomputation} -- \S\ref{sec:post processing}, as we feel it may be useful for those interested in pursuing similar large-scale distributed computations. \S\ref{sec:conjectures} contains our main experimental results, including conjectures on dominant schur modules \S\ref{subsec:dominant}) evidence for the normal distribution conjecture \S\ref{subsec:normal dist}, discussion of Boij-S\"oderberg coefficients \S\ref{subsec:BS coeffs}, unimodality conjectures \S\ref{subsec:unimodal}, and a discussion of the redundancy of Betti numbers \S\ref{subsec:redundancy}.

\section*{Acknowledgments}
First of all, we thank Rob Lazarsfeld, as he helped inspire us to pursue this project in the first place. We thank Claudiu Raicu, for advice and code related to the Schur functor computations. We thank Michael Saunders and Stephen Wright for their helpful advice on using numerical methods to perform rank computations. We also Lawrence Ein, David Eisenbud, Frank-Olaf Schreyer, Mike Stillman, and the many other people who offered suggestions about this work.

This research was performed using the compute resources and assistance
of the UW-Madison Center For High Throughput Computing (CHTC) in the
Department of Computer Sciences~\cite{CHTC}. The CHTC is supported by UW-Madison,
the Advanced Computing Initiative, the Wisconsin Alumni Research
Foundation, the Wisconsin Institutes for Discovery, and the National
Science Foundation, and is an active member of the Open Science Grid,
which is supported by the National Science Foundation and the
U.S. Department of Energy's Office of Science.

\section{Mathematical Background}\label{sec:background}

\subsection{Betti Number Notation}
Notation for Betti numbers can be confusing, so we outline our notation and discuss how it relates to other common notations. Throughout $S=\CC[x_0,x_1,\dots,x_n]$. Our computations center on the case $n=2$ and thus in \S\ref{sec:computational approach}--\S\ref{sec:post processing} we restrict to the case $n=2$. Given some $d\geq 1$ we let $R=\Sym(S_d)$ be the symmetric algebra, which is a polynomial ring on $\dim S_d$ many variables. While $R$ depends on the choice of $d$, we often abuse notation and omit reference to $d$.

We use $S(0;d)$ to denote the $d$th Veronese subring $S^{(d)}=\oplus_{i} S_{di}\subseteq S$, and we view $S(0;d)$ as an $R$-module. The ring $S(0;d)$ is the homogeneous coordinate ring of the image of $\PP^n$ under $d$-uple embedding $\iota\colon \PP^n\xhookrightarrow{} \PP^{\binom{n+d}{d}-1}$. We set $S(b;d):=\oplus_{i} S_{b+di}$, which is the graded $R$-module corresponding to the pushforward $\iota_*\cO_{\PP^n}(b)$. 

For the standard graded structure, we set $K_{p,q}(\PP^n, b;d)=\Tor_p(S(b;d), \CC)_{p+q}$. Using standard Betti number notation, we have $\beta_{p,p+q}(\PP^n,b;d)=\beta_{p,p+q}(S(b;d))=\dim K_{p,q}(\PP^n, b;d)$. The \defi{Betti table} of $(n,b;d)$ is then the table where $\beta_{p,p+q}(S(b;d))$ is placed in the $(p,q)$-spot. 

For the multigraded structure, we write $K_{p,q}(\PP^n, b;d)_\ba=\Tor_p(S(b;d), \CC)_\ba$ where $\ba\in \ZZ^{n+1}$ is a multidegree. In this notation we must have that $|\ba|=d(p+q)+b$. We also use the Betti number notation $\beta_{p,\ba}(\PP^n,b;d)=\beta_{p,\ba}(S(b;d))=\dim K_{p,q}(\PP^n, b;d)_\ba$. Note the standard graded Betti numbers are recoverable from the multigraded Betti numbers via the equation
\[
\beta_{p,p+q}(\PP^n,b;d) = \sum_{\substack{\ba\in \ZZ^{n+1}\\ |\ba|=(p+q)d+b}} \beta_{p,\ba}(\PP^n,b;d).
\]
A useful way to keep track of the multigraded Betti numbers is via the $\ZZ^{n+1}$-graded Hilbert series. In general if $M$ is a $\ZZ^{n+1}$-graded module then we use $\HS_M(t_0,t_1,\dots,t_{n})$ to denote the multigraded Hilbert series. This is particularly convenient for encoding the multigraded structure of the multigraded vector space $K_{p,q}(\PP^n, b;d)$, as we can write
\[
HS_{K_{p,q}(\PP^n, b;d)}(t_0,t_1,\dots,t_n) = \sum_{\substack{\ba \in \ZZ^{n+1} \\ |\ba|= d(p+q)+b}} \beta_{p,\ba}(S(b;d))t^\ba
\]
where if $\ba=(a_0,a_1,\dots,a_n)$ then $t^{\ba}:=t_0^{a_0}t_1^{a_1}\cdots t_n^{a_n}$.

\begin{remark}
Since we will only consider the case $n=2$ for much of the paper, we often write $K_{p,q}(b;d):=K_{p,q}(\PP^2,b;d)$. We similarly abbreviate the notation for the multigraded Betti numbers in the cases where we are working with $\PP^2$.
\end{remark}

These notions of Betti numbers, and the relations between them, are perhaps best understood through an example. 

\begin{example}\label{ex:K11 d3}
Consider $\PP^2\subseteq \PP^{9}$ embedded by $\cO_{\PP^2}(3)$. The Betti table of $S(0;3)$ is
\[
\begin{matrix}
1&-&-&-&-&-&-&-\\
-&{\mathbf{27}}&105&189&189&105&27&-\\
-&-&-&-&-&-&-&1\\
\end{matrix}.
\]
Focusing on the boldfaced $27$, we have $K_{1,1}(0;3)=\CC^{27}$ and $\beta_{1,2}(0;3)=27$. As a $\ZZ^3$-graded vector space, $K_{1,1}(0;3)$ has 19 distinct multidegrees, which we encode via the Hilbert series
\[
\HS_{K_{1,1}(0;3)}(t_0,t_1,t_2)= \begin{array}{l}{t}_{0}^{4} {t}_{1}^{2}+{t}_{0}^{3} {t}_{1}^{3}+{t}_{0}^{2} {t}_{1}^{4}+{t}_{0}^{4} {t}_{1} {t}_{2}+2 {t}_{0}^{3}
   {t}_{1}^{2} {t}_{2}+2 {t}_{0}^{2} {t}_{1}^{3} {t}_{2}+{t}_{0} {t}_{1}^{4} {t}_{2}+{t}_{0}^{4} {t}_{2}^{2}+2
   {t}_{0}^{3} {t}_{1} {t}_{2}^{2}\\
   +3 {t}_{0}^{2} {t}_{1}^{2} {t}_{2}^{2}
   +2 {t}_{0} {t}_{1}^{3}
   {t}_{2}^{2}+{t}_{1}^{4} {t}_{2}^{2}+{t}_{0}^{3} {t}_{2}^{3}+2 {t}_{0}^{2} {t}_{1} {t}_{2}^{3}+2 {t}_{0}
   {t}_{1}^{2} {t}_{2}^{3}+{t}_{1}^{3} {t}_{2}^{3}+{t}_{0}^{2} {t}_{2}^{4}\\+{t}_{0} {t}_{1} {t}_{2}^{4}+{t}_{1}^{2}
   {t}_{2}^{4}.
   \end{array}
\]
Thus for instance $K_{1,1}(0;3)_{(4,2,0)}=\CC$ and $K_{1,1}(0;3)_{(2,2,2)}=\CC^3$.
\end{example}

\subsection{Schur Modules and Dominant Weights}

We also consider the Schur functor decomposition of $K_{p,q}(\PP^n, b;d)$ arising from the linear action of $\GL_{n+1}(\CC)$ on $S$, and so briefly review the relevant notation and terminology. See~\cite{fulton-harris} for a review of this material.

If $\lambda=(\lambda_0\geq \lambda_2\geq \cdots \geq \lambda_n)$ is a partition of weight $|\lambda|=\lambda_0+\lambda_1+\cdots+\lambda_n$ we write $\bS_{\lambda}=\bS_{\lambda}(\CC^{n+1})$ for the corresponding Schur functor, which is a representation of $\GL_{n+1}(\CC)$. The Schur functor decomposition of $K_{p,q}(\PP^n, b;d)$ can be expressed as:
\[
K_{p,q}(\PP^n, b;d) = \bigoplus_{\substack{|\lambda|=d(p+q)+b}} \bS_{\lambda}(\CC^{n+1})^{\oplus m_{p,\lambda}(\PP^n,b;d)},
\]
with the $m_{p,\lambda}(\PP^n,b;d)=m_{p,\lambda}(n,b;d)$ being the Schur functor multiplicities. The Schur functor decomposition is recoverable from the multigraded Betti numbers (see Algorithm~\ref{alg:post process schur}).

Given $\lambda=(\lambda_1,\lambda_2,\dots,\lambda_a)$ and $\lambda'= (\lambda'_1,\lambda'_2,\dots,\lambda'_b)$ we say that $\lambda$ \defi{dominates} $\lambda'$ (or $\lambda \succeq \lambda'$) if $\sum_{i=1}^k \lambda_i \geq \sum_{i=1}^k \lambda'_k$ for all $k\geq \max\{a,b\}$. This induces a partial order on $\ZZ^{n+1}$, and given a subset $W\subset \ZZ^{n+1}$ we often write $\dw W$ for the set of dominant weights in $W$.

\begin{example}
Consider the Schur functor decomposition of $K_{14,1}(5;0)$, which appears in Appendix~2: 
\[
K_{14,1}(5;0)\cong{\bS_{(34,21,20)}}\oplus{\bS_{(33,25,17)}}\oplus\bS_{(33,24,18)}\oplus \cdots
\]
The weight $(33,24,18)$ is dominated by $(33,25,17)$ but is not dominated by $(34,21,20)$.  In this case, the two maximally dominant weights are $(34,21,20)$ $(33,25,17)$.
\end{example}

\subsection{Monomial syzygies}
In~\cite{eel-quick}, Ein, Erman, and Lazarsfeld use monomial techniques to construct nonzero elements of $K_{p,q}(\PP^n, b;d)$ for a wide range of values of the parameters. The basic idea behind the construction is the following: First, one replaces the Veronese ring $S^{(d)}=S(0;d)$ by the Veronese of a quotient $\overline{S}(0;d):=(S/(x_0^d, x_1^d,\dots,x_n^d))^{(d)}$. Writing $\overline{S_d}$ for the quotient vector space $S_d/((x_0^d, x_1^d,\dots,x_n^d)$, a standard Artinian reduction argument induces a natural isomorphism between the syzygies of $\overline{S}(0;d)$, resolved over $\Sym(\overline{S_d})$ and the syzygies of $S^{(d)}$ resolved over $\Sym(S_d)$. A similar statement holds for $S(b;d)$, and thus to produce nonzero elements of $K_{p,q}(\PP^n, b;d)$ it is enough to produce nonzero syzygies of $\overline{S}(b;d)$.

Ein, Erman, and Lazarsfeld produce monomial syzygies via the following recipe: for some degree $e$ let $f$ be the lex-leading monomial of degree $e$ that is nonzero in $\overline{R}$. For instance, if $e=d$ then we would take $f=x_0^{d-1}x_1$. Next, let $m_1,\dots,m_s$ be distinct monomials in $\overline{R}$ such that $m_if=0$ in $\overline{R}$ for all $1\leq i \leq s$. For instance, $m_i$ could be any monomial divisible by $x_0$. Let $\zeta:= m_1\wedge m_2\wedge \cdots \wedge m_s \otimes \in \bigwedge^s \overline{S_d}\otimes \overline{S_d}$. Then $\zeta$ will induce a cycle in the appropriate sequence of the form~\eqref{eqn:3 term}. Under mild restrictions on the $m_i$, one also shows that $\zeta$ is not a boundary, and hence it induces a nonzero element the homology group $K_{p,q}(\PP^n, b;d)$.

We write $\mathrm E_{p,q}(\PP^n, b;d)$ for the vector space of monomial syzygies constructed in~\cite{eel-quick}. This is a $\ZZ^{n+1}$-graded vector space, and hence we can also discuss the dominant weights of this space, which we denote by $\dw \mathrm E_{p,q}(\PP^n, b;d)$.

\section{Overview of computational Approach}\label{sec:computational approach}
For our computations, we focus on the case of $\PP^2$. The standard computational approach involves computing a minimal free resolution for $M$, but the complexity grows quite quickly with $d$. For instance $d=2,3,4$ are easily computable in Macaulay2, but our computation did not terminate for $d= 5$. 

We take a different approach, relying on linear algebra computations to determine the Betti table. 
Using the Koszul complex to compute $\Tor$-groups, we have that the vector space $K_{p,q}(b;d)$ is the cohomology of the complex:
\begin{equation}\label{eqn:3 term}
\begin{tikzcd}[column sep = 3em]
\bigwedge^{p+1}S_d\otimes S_{b+(q-1)d}\rar{\partial_{p+1}}& \bigwedge^{p}S_d\otimes S_{b+qd}\rar{\partial_{p}}&\bigwedge^{p-1}S_d\otimes S_{b+(q+1)d}
\end{tikzcd}
\end{equation}
where $\partial_{p}$ is defined by:
\[
\partial_p\left(m_1\wedge m_2\wedge\cdots\wedge m_p\otimes f\right)=\sum_{k=1}^p(-1)^km_1\wedge m_2\wedge\cdots\wedge\widehat{m}_k\wedge\cdots\wedge m_p\otimes (m_kf).
\]
Since the differential respects the $\ZZ^3$-multigrading, it suffices to separate \eqref{eqn:3 term} into multigraded strands. Thus for each $\mathbf{a}=(a_0,a_1,a_2)\in \ZZ^3$ where $a_0+a_1+a_2=b+d(p+q)$, we have that $K_{p,q}(b;d)_{\ba}$ is the cohomology of
\begin{equation}\label{eqn:3 term multigraded}
\begin{tikzcd}[column sep = 3em]
\left(\bigwedge^{p+1}S_d\otimes S_{b+(q-1)d}\right)_{\mathbf{a}}\rar{\partial_{p+1,\mathbf{a}}}& \left(\bigwedge^{p}S_d\otimes S_{b+qd}\right)_{\mathbf{a}}\rar{\partial_{p,\mathbf{a}}}&\left(\bigwedge^{p-1}S_d\otimes S_{b+(q+1)d}\right)_{\mathbf{a}}.
\end{tikzcd}
\end{equation}
This reduces computing the multigraded (or graded) Betti numbers of $M$ to the computation of a large number of individual matrices. Note if we choose bases for the source and target consisting of monomials each $\partial_{p,\ba}$, is represented matrices whose entries are either $0$ or $\pm 1$. 

\begin{example}
Consider $K_{2,2}(0;3)_{(7,3,2)}$, which is one of the multigraded entries for the structure sheaf $\cO_{\PP^2}$ under the $3$-uple embedding. To compute this, we first construct the matrix $\partial_{2,(7,3,2)}$. We choose products of monomials for a basis on both the source and target. For instance $x^3\wedge x^2y \otimes x^2y^2z^2\in \left(\bigwedge^2 S_3 \otimes S_3\right)_{(7,3,2)}$ is a basis vector in the source. We have
\[
\partial_{2,(7,3,2)}(x^3\wedge x^2y \otimes x^2y^2z^2) = x^3\otimes x^4y^3z^2 - x^2y\otimes x^5y^2z^2.
\]
Working over all such monomial, we represent $\partial_{2,(7,3,2)}$ by a matrix
\[
\bordermatrix{& &x^3\wedge x^2y\otimes x^2y^2z^2 \ \ & \ \ x^3\wedge xy^2\otimes x^3yz^2 \ \ & \\ x^3\wedge x^2z\otimes x^2y^3z \\ & \cdots \cr
x^3\otimes x^4y^3z^2 &&1&1&1&\cdots\cr
 x^2y\otimes x^5y^2z^2 &&-1&0&0&\cdots\cr
 x^2z\otimes x^5y^3z&&0&0&-1&\cdots\cr   
 xy^2\otimes x^6y^2z^2 &&0&-1&0&\cdots\cr
         && \vdots&\vdots&\vdots&\ddots \cr
}.
\]
The dimension of $K_{2,2}(0;3)_{(7,3,2)}$ is determined by the ranks and sizes of these matrices.  Since $\partial_{2,(7,3,2)}$ has $23$ columns, we have
\begin{align*}
\dim K_{2,2}(0;3)_{(7,3,2)} &= \dim \ker \partial_{2,(7,3,2)}-\dim \im \partial_{3,(7,3,2)}\\
&=\left(23 -\rank \partial_{2,(7,3,2)}\right) - \rank \partial_{3,(7,3,2)}\\
&=23 -8 -15 =0.
\end{align*}
\end{example}

Our computational approach can be summarized as follows:
\begin{enumerate}[noitemsep]
	\item {\bf Precomputation:} We use known vanishing results and facts about Hilbert series to reduce the number of matrices whose rank we need to compute. We also use standard duality results to focus on simpler matrices in some cases. 
	\item {\bf Main computation:} We construct the remaining relevant matrices, and use an LU-decomposition algorithm and distributed, high throughput computations to compute the ranks of those matrices.
	\item {\bf Post-processing:} We assemble our data to produce the total the multigraded Betti numbers and the standard Betti numbers , and we apply a highest weight decomposition algorithm to obtain the Schur module decompositions.
\end{enumerate}
While the largest computational challenges come from the main computation step, the scale of our data creates some challenges in executing and coordinating the other steps. In the following sections, we describe the relevant issues in some detail.

\begin{remark}
With the exception of the rank computations, all other steps are symbolic in nature. However, since we use a numerical algorithm to compute the ranks of the matrices, there is potential for numerical error in that step. See \S\ref{subsec:QR}. In post-processing, we decompose the $K_{p,q}$ space into Schur modules, and this can correct small numerical errors. See \S\ref{subsec:numerical error}.
\end{remark}

\section{Precomputation}\label{sec:precomputation}
\subsection{Determining the relevant range of Betti numbers}\label{subsec:relevant range}
A number of the $K_{p,q}$ spaces are entirely determined by combining the $\ZZ^3$-graded Hilbert series with known vanishing results. The following lemma is well-known to experts, but we include for reference. 

\begin{lemma}\label{lem:rational function}
The $\ZZ^3$-graded Hilbert series for $S(b;d)$ is a rational function of the form $A(t_0,t_1,t_2)/B(t_0,t_1,t_2)$ where
\[
A(t_0,t_1,t_2) = \sum_{p,\ba} (-1)^p\dim K_{p,q}(b;d)_{\ba}t^{\ba} \text{ and } B(t_0,t_1,t_2) = \prod_{\bb\in \mathbb N^3, |\bb|=d} 1-t^\bb.
\]
\end{lemma}
\begin{proof}
Let $\bF$ be the minimal free resolution of $S(b;d)$ as a $R:=\Sym(S_d)$-module. The ring $R=\Sym S_d$ inherits a natural $
\ZZ^3$-grading from the $\ZZ^3$-grading on $S_d$, and thus we can assume that $\bF=[\dots \to F_1\to F_0]$ is $\ZZ^3$-graded. If we write $F_p:=\oplus_{\ba} R(-\ba)^{\beta_{p,\ba}}$, then we have $\beta_{p,\ba}=\dim K_{p,q}(b;d)_{\ba}$.

The $\ZZ^3$-graded Hilbert series of $R$ is $1/B(t_0,t_1,t_2)$ and thus the Hilbert series of $R(-\ba)$ is $t^\ba/B(t_0,t_1,t_2)$. The desired statement then follows from additivity of Hilbert series.
\end{proof}

If we fix $p,b,d$ and $|\ba|$, then Lemma~\ref{lem:rational function} implies that $A(t_0,t_1,t_2)$ entirely determines $K_{p,q}(b;d)_{\ba}$, unless there are multiple values of $q$ such that $K_{p,q}(b;d)_{\ba}\ne 0$. 

\begin{defn}
Given $b$ and $d$, we define the \defi{relevant range} as the set of pairs $(p,q)$ where $K_{p,q}(b;d)\ne 0$ and where either $K_{p-1,q+1}(b;d)\ne 0$ or $K_{p+1,q-1}(b;d)\ne 0$.
\end{defn}

For instance, looking at $\beta(\PP^2, 0;5)$ in Appendix 1, we see that the relevant range is the set $\{(14,1),(15,1),(13,2),(14,2)\}$. All other entries are determined by the Hilbert series. 

Since it easy to compute the Hilbert series of the modules $S(b;d)$, it will be much easier to compute Betti numbers outside of the relevant range. For $\PP^2$ the relevant range is precisely understood. See~\cite[Remark~6.5]{ein-lazarsfeld} for the $K_{p,0}$ and $K_{p,2}$ statements, and ~\cite[Theorem~2.2]{green-II} and \cite[Theorem~2.c.6]{green-I} for the $K_{p,1}$ statements.

\subsection{Computing outside the relevant range}
For values outside of the relevant range, we compute $\beta_{p,\ba}(b;d)$ using the Hilbert series. Recall that we must have $|\ba|=(p+q)d+b$ for the space to be nonzero. The following elementary algorithm computes $A(t_0,t_1,t_2)$.
\begin{algorithm} \ 

\begin{tabular}{l l}
	 Input:&$b,d$. \\
	Output: &The polynomial $A(t_0,t_1,t_2)$ for $S(b;d)$, as in Lemma~\ref{lem:rational function}\\
	&- $N:=d\left(\binom{d+2}{2}-1\right)$ and $L:=\{\ba \in \mathbb N^3 | |\ba| \equiv b (\text{mod } d) \text{ and } |\ba| \leq N\}$.\\
	& - $C(t_0,t_1,t_2) := \sum_{\ba \in L} \dim S_{\ba}t^\ba$\\
	& - Let $A(t_0,t_1,t_2)$ be the sum of all terms of degree $\leq N$ in the product of\\
	& \phantom{- }$C(t_0,t_1,t_2)$ and $B(t_0,t_1,t_2)$.
\end{tabular}
\end{algorithm}
\begin{proof}
By Lemma~\ref{lem:rational function} we have $A(t_0,t_1,t_2)=\HS_{S(b;d)}(t_0,t_1,t_2)B(t_0,t_1,t_2)$.  If we can bound the degree of $A(t_0,t_1,t_2)$, then we can bound the number of terms in the power series $\HS_{S(b;d)}$ that we will need to consider. Each $S(b;d)$ is a Cohen-Macaulay $R$-module and thus has projective dimension $\binom{d+2}{2}-3$. And since $0\leq b \leq d$, the regularity of $S(b;d)$ is at most $2$. Thus, the largest total degree of a nonzero Betti number of $S(b;d)$ is $N=d\left(\binom{d+2}{2}-1\right)$, and it follows that $\deg A \leq N$. By definition, $C(t_0,t_1,t_2)$ is the sum of all terms of degree $\leq N$ in the power series $\HS_{S(b;d)}(t_0,t_1,t_2)$.
\end{proof}

\section{Main Computation}\label{sec:main computation}
\subsection{Constructing the matrices in the relevant range}
Within the relevant range we can incorporate the $S_3$-symmetry to restrict to multidegrees $(a_0,a_1,a_2)$ where $a_0\geq a_1\geq a_2$. Moreover, we can use duality for Koszul cohomology groups to further cut down the number of matrices we need to compute~\cite[Theorem~2.c.6]{green-I}. The table in Figure~\ref{fig:matrices} lists the number of matrices needed (after accounting for duality and $S_3$-symmetries) in the computations for various values of $d$ and $b$.
\begin{figure}
\[
\begin{array}{c|c|c|c}
&&\# \text{ of Relevant} & \text{Largest} \\ 
d&b& \text{Matrices}& \text{Matrix}\\\hline
\multirow{4}{*}{4} &0& 0 & \text{N/A}\\ \cline{2-4}
&1&0 &\text{N/A}\\ \cline{2-4}
&2&56 &255\times 669\\ \cline{2-4}
&3&56 &255\times 669\\ \hline

\end{array}
\qquad
\qquad
\begin{array}{c|c|c|c}
&&\# \text{ of Relevant} & \text{Largest} \\ 
d&b& \text{Matrices}& \text{Matrix}\\\hline
\multirow{5}{*}{5}&0& 102 & 2,151 \times 3,159\\ \cline{2-4}
&1&0 &\text{N/A}\\ \cline{2-4}
&2&102 &2,151 \times 3,159\\ \cline{2-4}
&3&424 &38,654\times 95,760\\ \cline{2-4}
&4&424 &38,654\times 95,760\\ \hline
\end{array}
\]
\caption{There are no relevant matrices for $d<4$. For $d=4,5$ and each $b$, we list the number of relevant matrices. See Figure~\ref{fig:d6} for the data when $d=6$. Some lines are identical because of duality of Koszul cohomology groups.}
\label{fig:matrices}
\end{figure}

\begin{remark}
For testing purposes, we also computed $\dim K_{p,q}(b;d)$ using our rank algorithms for many $(p,q)$ outside of the relevant range, including all $K_{p,q}(b;d)(b;d)$ for $d\leq 4$. In all cases, the computation gave the correct result.
\end{remark}

Writing out and storing each of the matrices $\partial_{p,\mathbf{a}}$ is inefficient, both in terms of runtime and memory. We streamline this process by utilizing a symmetry of the matrices $\partial_{p,\ba}$ for various $\ba$. Consider the commutative diagram
\[
\xymatrix{
 \bigwedge^p S_d\otimes S_{qd+b}\ar[rr]^-{\partial_p}\ar[d]^-{\phi_p}&& \bigwedge^{p-1} S_d\otimes S_{q(d+1)+b}\ar[d]^-{\phi_{p-1}}\\
 \bigwedge^p S_d\ar[rr]^-{d_p}&& \bigwedge^{p-1} S_d
}
\] 
where $\phi_p(m_1\wedge m_2\wedge \cdots \wedge m_p \otimes f)=m_1\wedge m_2\wedge \cdots \wedge m_p$ and similarly for $\phi_{p-1}$, and where $d_p\left(m_1\wedge m_2\wedge\cdots\wedge m_p\right)=\sum_{k=1}^p(-1)^km_1\wedge m_2\wedge\cdots\wedge\widehat{m}_k\wedge\cdots\wedge m_p$. 

We represent $\partial_p$ and $d_p$ as matrices with respect to the basis consisting of wedge/tensor powers of all monomials. For a multidegree $\ba$, we write $d_{p,\leq \ba}$ for the submatrix of $d_p$ involving basis vectors of degree $\leq \ba$. 

We claim that $\partial_{p,\ba}$ equals $d_{p,\leq \ba}$. The crucial observation is the following. For a pure tensor $m_1\wedge m_2\wedge \cdots \wedge m_p \otimes f\in \bigwedge^p S_d \otimes S_{e}$ of degree $\ba$, the monomial $f$ is entirely determined by the multidegree $\ba$ and by the monomials $m_1,m_2,\dots, m_p$. In other words, if we know the multidegree of a monomial pure tensor, then the $\otimes S_e$ factor is redundant information.

\begin{example}
Let $S=\CC[x,y,z]$ and 
consider a monomial $m_1\wedge m_2\otimes f\in (\bigwedge^2 S_3\otimes S_3)_{(7,3,2)}$.  If $m_1=x^3$ and $x_2=x^2y$ then $f$ must equal $x^2y^2z^2$.
\end{example}

We can thus compute the rank of $\partial_{p,\ba}$ by computing the rank of a submatrix of a $d_p$. So instead of constructing and storing each $\partial_{p,\ba}$, we simply precompute the matrix $d_p$ and then take slices corresponding to any particular multidegree. In practice, this seemed to significantly improve the runtime and memory on the construction of the matrices, though we did not track precise comparisons with a more naive construction of the matrices.

\begin{example}
For $d=6$, it took one hour on a standard laptop to construct the relevant matrices for all $b$. One of the more complicated entries, $K_{9,0}(3;6)$, required 178 distinct matrices which took up a total of 2 GB of space. While the bulk of these matrices are very small, some of the matrices can be massive. See Figure~\ref{fig:matrices} and Example~\ref{ex:RAM} for more details.
\end{example}

\subsection{Sparse linear algebra}\label{subsec:QR}
Computing the cohomology of \eqref{eqn:3 term multigraded} amounts to computing the ranks of many matrices. However, as seen in Figure~\ref{fig:matrices} these matrices can be quite large. While standard (dense) matrix algorithms quickly fail to terminate, the matrices turn out to be quite sparse, as the formula for $\partial_p$ given in equation~\ref{eqn:3 term} implies that each row of $\partial_{p,\mathbf{a}}$ has only $p$ distinct entries.

\begin{example}
For $K_{8,1}(0;5)$ we use $41$ matrices which range in size from $23\times 144$ to $22,349\times 24,157$. For the largest these matrices, only $0.03\%$ of the entries are nonzero.
\end{example}

We can thus use sparse algorithms for our rank computations. Specifically, we base our rank computation on a rank revealing version of LU-factorization. Like many matrix factorizations, LU-factorization seeks to write a matrix as product of two matrices that are easier to understand. In particular, if $A$ is an $m\times n$ matrix with $m\geq n$, then an exact LU-factorization writes $A$ as:
\[
QAP=LU=L\begin{pmatrix}
U_{11} & U_{12} \\ 
0 & 0
\end{pmatrix}
\]
where: 
\begin{itemize}[noitemsep]
\item $Q$ is an $m\times m$ permutation matrix,
\item $P$ is a $n\times n$ permutation matrix,
\item $L$ is a lower triangular matrix with unit diagonal,
\item $U_{11}$ is a non-singular $r\times r$ upper-triangular matrix, and
\item $U_{12}$ is a $r\times(m-r)$ matrix.
\end{itemize}
Given such a factorization the rank of $A$ equals the size of $U_{11}$. Since we use numerical computation, in practice we generally factor $A$ as:
\[
QAP=LU=L\begin{pmatrix}
U_{11} & U_{12} \\ 
0 & U_{22}
\end{pmatrix}
\]
where $P,Q,L,U_{11}, U_{12}$ are as before, and $U_{22}$ is, in a sense, insignificant relative to $U_{11}$. Specifically we want the smallest singular value of $U_{11}$ to be much bigger than the largest singular value of $U_{22}$. The number of non-zero singular values of $A$ is then approximately the number of non-zero singular values of $U_{11}$ i.e. the size of $U_{11}$. (For a more in-depth discussion of using LU factorizations see \cite[Section 3.2]{golub-matrix-computations}.)

We use the MatLab interface to the {\verb LUSOL } library~\cite{LUSOL,ML} to produce an LU-factorization of each matrix $A:= \partial_{p,\ba}$. 
The matrix $U$ is an $n\times m$ upper triangular matrix whose diagonal entries are decreasing in size, and $U_{11}$ is the sub-matrix of $U$ on the first $k$ columns and rows where $k$ is the largest number such that that $|U_{k,k}|$ is greater than some chosen tolerance. 
(Since we do not have clear data on how to appropriately chose the tolerance, we chose instead to vary the tolerance.  See Remark~\ref{rmk:tolerance} below.)  In this way we actually compute the numerical rank of $A$ with respect to this given tolerance. 
More succinctly we approximate the rank of the differential $A:= \partial_{p,\ba}$ by counting the number of diagonal entries of $U$ larger than the above tolerance. 

While this approach using sparse LU-factorization algorithms allows us to go beyond what is currently possible with dense matrix algorithms there are two down sides. First, while spares, our matrices tend to have very high rank (relative to their size). For instance, the matrix $\partial_{9,(19,19,19)}$ has size $596,898 \times 1,246,254$ and rank $596,307$. This adds to the complexity of the LU-factorization algorithm, and increases runtime and memory usage. Second, due to the threshold and the approximate nature of our factorization, our rank computations are numerical and not symbolic in nature. There is the possibility of numerical error. See \S\ref{subsec:numerical error} for a discussion of error, and how post-processing catches some small numerical errors.

\subsection{High Throughput Computations}
The rank computations can be efficiently distributed over numerous different computers. We implemented these computations using high throughput computing via HTCondor~\cite{HTCondor} on both the University of Wisconsin-Madison campus computing pool~\cite{CHTC} and on the Open Science Grid~\cite{OpenScienceGrid}. In addition, since some of those rank computations require substantial RAM, we make use of University of Wisconsin-Madison's High Throughput Computing cluster to manage our computations.

We do not a priori know the RAM and time required for individual matrix computations. We thus start by submitting jobs with LUSOL's default memory allocation. For the jobs that fail, we increase the memory allocation and resubmit. We iterate this process until the computation terminates. This approach works well with our data, as the bulk of of
the matrices terminate with very little RAM. However, for some computations, the memory
required exceeds 250 GB of RAM, and our hardware grid has a small number of nodes
with this much RAM available. These largest computations can take days to complete

\begin{example}\label{ex:RAM}
One of our larger computations was for the Betti number $K_{9,0}(3;6)$. After accounting for symmetries, the computation involved $178$ distinct matrices, the largest of which was $596,898\times 1,246,254$. For the matrices, the RAM and time used were:
\begin{itemize}[noitemsep]
	\item 80\% used $<1$ GB RAM, taking $<1$ minute on average, with a max of 18 minutes.
	\item 9\% used 1-10 GB RAM, taking 13 minutes on average, with a max of 40 minutes.
	\item 10\% 10-100 GB RAM, taking 5 hours on average, with a max of 15 hours.
	\item The remaining two matrices each used $450$ GB RAM. One took 27 hours and the other took 49 hours. 
\end{itemize}
\end{example}

\section{Post-processing the Data}\label{sec:post processing}
\subsection{Betti numbers and Schur module decompositions}
Finally, we assemble and post-process the data.  Obtaining the multigraded Betti numbers and the total Betti numbers is simple. For the multigraded and total Betti numbers, we have
\[
\beta_{p,\ba}(b;d) = \rank(\ker \partial_{p,\mathbf{a}}) - \rank(\partial_{p+1,\mathbf{a}}) \quad \text{ and } \quad
\beta_{p,p+q}(b;d) = \sum_{\substack{\ba\in \ZZ^{n+1}\\ |\ba|=(p+q)d+b}} \beta_{p,\ba}(b;d)
\]
respectively.
%
We determine the Schur module decomposition via the highest weight greedy algorithm below. For a polynomial $P$ we write $\lex(P)$ for the lex-leading monomial of $P$. 
\begin{algorithm}[Schur Module Decomposition]\label{alg:post process schur} \ 

\begin{tabular}{l l}
Input: &$\beta_{p,\ba}(b;d)$ for fixed $b,d,p$ and all $\ba$ with $|\ba|=(p+q)d+b$.\\
Output: &A list $K$ of the partitions appearing in the Schur module decomposition of \\
&$K_{p,q}(b;d)$, with multiplicity.\\
&- $L:=\{ \ba | |\ba|=(p+q)d+b\}$ and $H=\sum_{\ba\in L} \beta_{p,\ba}(b;d)\cdot t^\ba$.\\
&- $K=\{\}$.\\
& - While the coefficient of $\lex(H)>0$ do:\\
& \qquad - Let $\lambda=(\lambda_1,\lambda_2,\lambda_3)$ be the weight of the lex-leading monomial in $H$.\\
& \qquad - Let $K=K\cup \{\lambda\}$.\\
&\qquad - Let $H$ equal $H$ minus the multigraded Hilbert series of the Schur module $\bS_\lambda(\CC^3)$.\\
&- Return $K$.
\end{tabular}
\end{algorithm}

\subsection{Numerical Error}\label{subsec:numerical error}
With the exception of the computation of the multigraded Betti numbers all other steps in our computation are symbolic. However, as we use numerical methods for the rank computations there is a chance for errors to appear in the multigraded Betti numbers.

We can sometimes detect and correct numerical errors when we apply Algorithm~\ref{alg:post process schur}. Since the $K_{p,q}(b;d)$ spaces are $\GL_3$-representations, each space must decompose as a direct sum of a much smaller number of Schur modules. Since the numerical errors tend to arise in the larger matrices, which tend to involve the most balanced weights, the greedy algorithm outlined in Algorithm~\ref{alg:post process schur} still sometimes suggests a ``best fit Schur decomposition'' for our multigraded data.

For all cases with $d\leq 5$, there appear to have been no numerical errors.\footnote{Interestingly, earlier computations where we used a QR-decomposition algorithm seem to have produced minor numerical errors in a small number of multigraded Betti numbers for $d=5$.} With $d=6$: there appear to be no numerical errors for $b=1,2$; there do appear to be some numerical errors for $b=0,3$ and we are continuing to process those results; and we are still awaiting the complete results for $b=4,5$, but we expect to find numerical errors in those cases as well. We believe that finding a more robust ``best fit Schur decompose'' algorithm will be crucial for extending our computation beyond $d=6$.

\begin{remark}\label{rmk:tolerance}
In a different direction, we also vary the tolerance in our computation as a way of understanding these numerical errors.  For many of the rank computations, we actually perform an array of computations with various different values for the tolerance.  This enables us to look over the data to see if a rank value is stable with respect to an array of tolerance values, as that would increase our confidence in the result.  Moreover, we hope that this data will provide a foundation for predicting appropriate tolerance values, and thus improving the algorithm in the future.  At the moment, this remains largely speculative.
\end{remark}

\section{Conjectures}\label{sec:conjectures}
In this section, we summarize several conjectures and observations derived by combining our data with other known results. 

\subsection{Dominant Schur Modules}\label{subsec:dominant}
The most efficient way to encode the structure of the Betti tables of Veroneses is via the Schur functor description, as this encapsulates and encodes the essential symmetries of these Betti tables.  We thus begin by focusing on Question~\ref{q:schur}.  When analyzing a representation of $\GL_n$, the first layer is, in a sense, given by the dominant weight representations that appear in the decomposition.  Our data led us to a conjecture about these dominant weight representations.  This can be viewed as a first approximation of an answer to Question~\ref{q:schur}.  Moreover, our proposed answer sharpens Ein and Lazarsfeld's Vanishing Conjecture for Veronese syzygies~\cite[Conjecture 7.5]{ein-lazarsfeld}, and it suggests a strong uniformity among the syzygies arising in each row of the Betti table.

In \S\ref{sec:background} above we reviewed the monomial syzygy construction from~\cite{eel-quick}. While the monomial syzygies $\mathrm E_{p,q}(\PP^n, b;d)$ represent only a small fraction of the total syzygies, they are conjecturally sufficient to give sharp vanishing/nonvanishing bounds~\cite[Remark~2.8]{eel-quick}. In other words, Ein and Lazarsfeld's conjecture on vanishing says that
\[
\mathrm E_{p,q}(\PP^n, b;d) \ne 0 \iff K_{p,q}(\PP^n,b;d) \ne 0.
\]
We conjecture that these monomial syzygies not only control the (non)vanishing of the $K_{p,q}(b;d)$, but they also determine the most dominant Schur module weights.
\begin{conj}\label{conj:domweights}
For all $n,d,b,p$ and $q$, we have:
\[
\dw \mathrm E_{p,q}(\PP^n, b;d) = \dw K_{p,q}(\PP^n, b;d).
\]
\end{conj}

We underscore the counterintuitive nature of the conjecture.
We see no obvious reason why monomial syzygies should determine the vanishing/nonvanishing question, let alone why they would provide a full description of the dominant weights. But Conjecture~\ref{conj:domweights}, which was discovered primarily through our experimental data, suggests that these simple monomial syzygies are deeply connected to the Schur module structure of the $K_{p,q}$ spaces.

\begin{example}
The space $K_{2,1}(0;4)$ is the direct sum of $9$ distinct Schur modules, each with multiplicity one. There are two dominant weight Schur modules: $\bS_{(9,2,1)}$ and $\bS_{(8,4,0)}$. These are naturally in bijection with the two dominant weight monomial syzygies from $\rE_{2,1}(0;4)$: $x_0^3x_1\wedge x_0^3x_2 \otimes x_0^3x_1$ and $x_0^3x_1\wedge x_0^2x_1^2 \otimes x_0^3x_1$.
\end{example}

The conjecture also suggests a mysterious uniformity among all of the $K_{p,q}(\PP^n,b;d)$ lying in a single row of a Betti table. Namely, if we vary only $p$, then the monomial syzygies constructed in~\cite{eel-quick} naturally form a graded lattice, with a unique maximal and minimal element. In other words, is is natural to think of the entire $q$th row as a single object
\[
K_{\bullet,q}(\PP^n,b;d) :=\bigoplus_{p} K_{p,q}(\PP^n,b;d)
\]
and to ask whether this vector space is a representation (or even an irreducible representation) over a larger group. Precisely such a phenomenon occurs when $d=2$ by~\cite{sam-super}.

\begin{example}\label{ex:lattice}
Consider $K_{\bullet,1}(\PP^2,0;3)$, which corresponds to the first row of the Betti table in Example~\ref{ex:K11 d3}. The most dominant weights of $K_{p,1}(\PP^2,0;3)$ are in bijection with the weights of the monomial syzygies in the $p$th row of the following lattice:
\[
\begin{tikzpicture}[yscale=1.1]
\draw(0,5) node{$x^2y\otimes x^2y$};
\draw(-2,4) node{$x^2y\wedge xy^2\otimes x^2y$};
\draw(2,4) node{$x^2y\wedge x^2z\otimes x^2y$};
\draw(0,3) node{$x^2y\wedge xy^2\wedge x^2z\otimes x^2y$};
\draw(0,2) node{$x^2y\wedge xy^2\wedge x^2z\wedge xyz\otimes x^2y$};
\draw(-4,1) node{$x^2y\wedge xy^2\wedge x^2z\wedge xyz\wedge y^2z \otimes x^2y$};
\draw(3.3,1) node{$x^2y\wedge xy^2\wedge x^2z\wedge xyz\wedge xz^2\otimes x^2y$};
\draw(0,0) node{$x^2y\wedge xy^2\wedge x^2z\wedge xyz\wedge y^2z \wedge xz^2\otimes x^2y$};
\draw[-,thick] (0,4.7)--(-2,4.3);
\draw[-,thick] (0,4.7)--(2,4.3);
\draw[-,thick] (0,3.3)--(-2,3.7);
\draw[-,thick] (0,3.3)--(2,3.7);
\draw[-,thick] (0,2.3)--(0,2.7);
\draw[-,thick] (0,1.7)--(-2,1.3);
\draw[-,thick] (0,1.7)--(2,1.3);
\draw[-,thick] (0,0.3)--(-2,0.7);
\draw[-,thick] (0,0.3)--(2,0.7);
\end{tikzpicture}
\]
\end{example}

We have confirmed Conjecture~\ref{conj:domweights} even in some cases where a full computation of the $K_{p,q}$ is infeasible. Our rank computations of $\partial_{p,\ba}$ take the longest when $\ba$ is highly balanced, e.g. $\ba = (12,11,10)$. By contrast, Conjecture~\ref{conj:domweights} addresses the most dominant -- and thus most unbalanced -- weights. 
The parallel nature of our computational techniques thus enabled us to verify Conjecture~\ref{conj:domweights} in some cases when $d=7$.

In all of the examples we have computed, the multiplicities of the dominant weight Schur modules is always one. It would be interesting to know whether this always holds.
\begin{question}\label{q:mult one}
Let $\lambda \in \dw K_{p,q}(\PP^n, b;d)$. Does the representation $\bS_\lambda(\CC^{n+1})$ appear in $K_{p,q}(\PP^n, b;d)$ with multiplicity one?
\end{question}
Related to Conjecture~\ref{conj:domweights} and Question~\ref{q:schur}, we propose the following vague question:
\begin{question}
Find a compelling combinatorial description of the set $\dw K_{p,q}(\PP^n, b;d)$.
\end{question}

A closely related conjecture based on the data is that the last nonzero $K_{p,1}(0;d)$ space is a particular irreducible Schur module.
\begin{conj}\label{conj:specific betti}
Let $p:= d \cdot \binom{d+1}{2}$, and let
\[
(a,b,c) = \bigg(\tbinom{d+2}{3}-1\ ,\  \tfrac{1}{6} d(d^2+5)\ ,\ \tbinom{d+1}{3}-1\bigg).
\]
We have
$
K_{p,p+1}(\PP^2,0;d)\cong \bS_{(a,b,c)}.
$
\end{conj}
This provides a Schur functor analogue of the corresponding conjecture from~\cite[\S8.3]{wcdl}. Note that the specific value of $p$ is the maximum value where $K_{p,p+1}(0;d)\ne 0$.

\subsection{Normal Distribution}\label{subsec:normal dist}
\begin{figure}
\includegraphics[scale=.35]{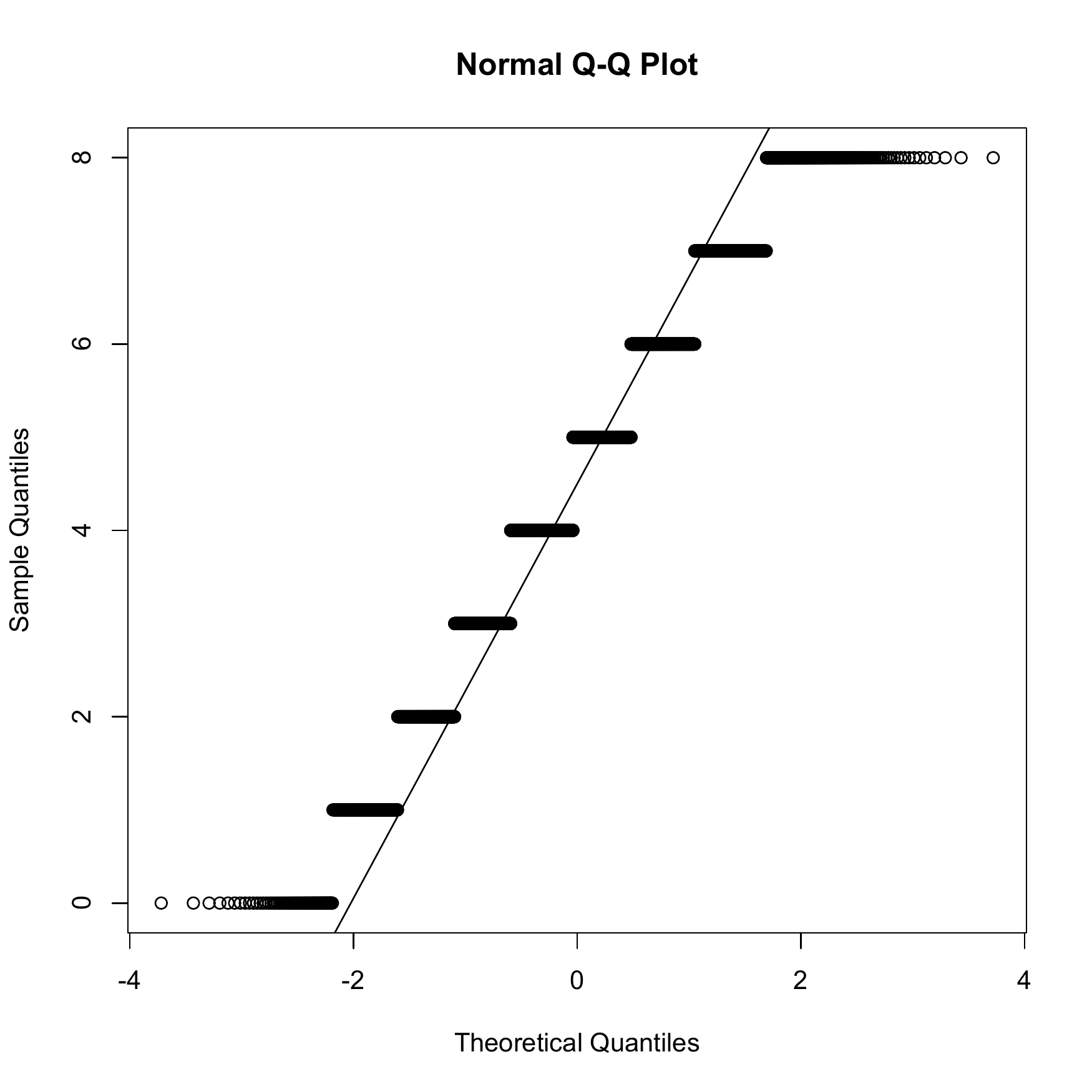}
\qquad
\includegraphics[scale=.37]{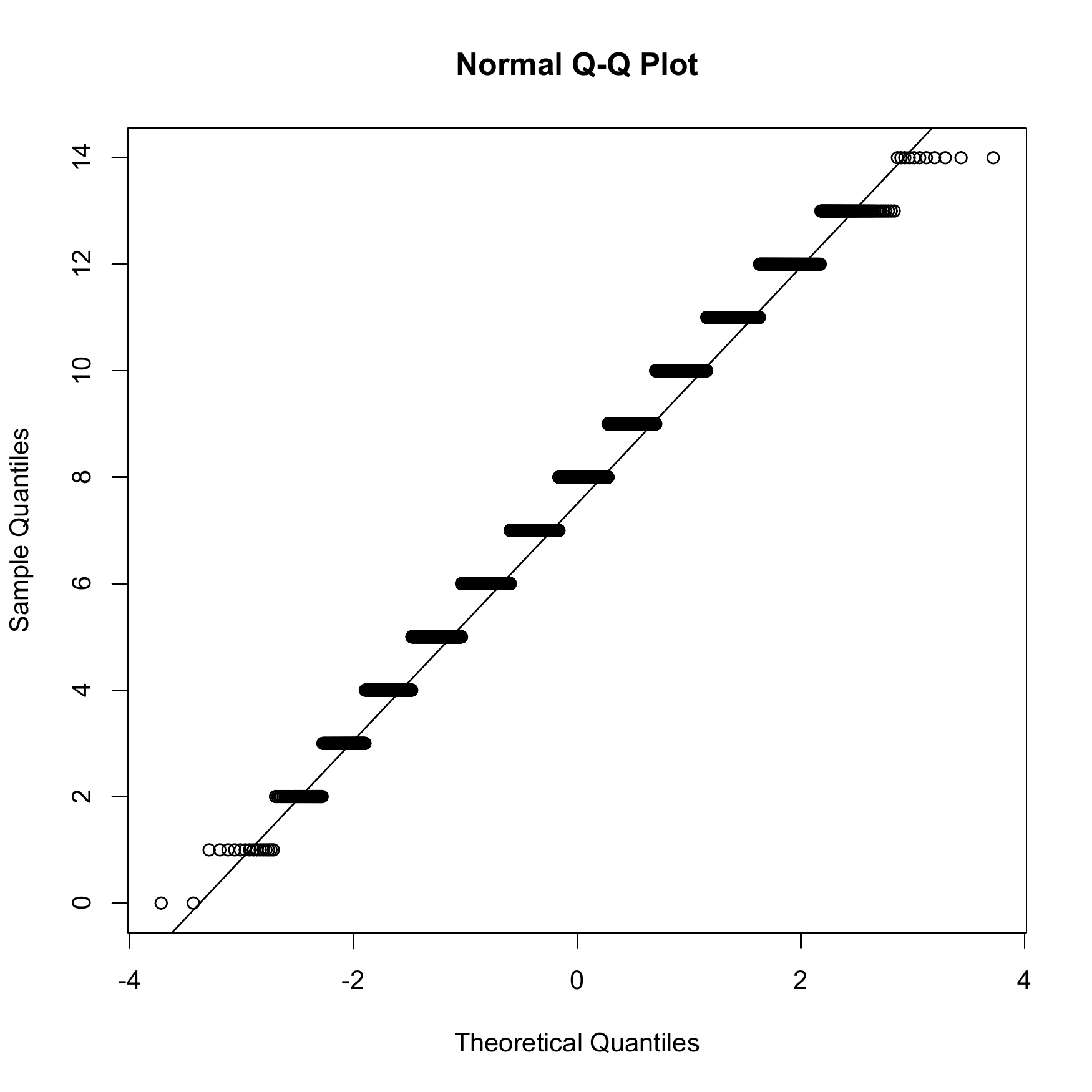}
\caption{These Q-Q plots are based on the Betti numbers in the first row of the Betti diagram of $\cO_{\PP^2}$ with respect to the $d=5$ (on the left) and $d=6$ embeddings. They suggest a normal distribution, as in~\cite[Conjecture~B]{eel-betti}.}
\label{fig:QQ}
\end{figure}

In \cite{eel-betti} the authors show that a ``randomly'' chosen Betti table converges (up to some rescaling function) to a binomial distribution, which then in turn converges to a normal distribution via the law of large numbers. This led to the conjecture that for Veronese embeddings, a plot of Betti numbers in any row of the Betti should converge, after rescaling, to a normal distribution.

In order to test this conjecture it is useful to define what we call the \defi{Betti distribution for $(q,b,d)$}. Fixing $b,d,$ and $q$, we consider the function $p\mapsto C \dim \cdot K_{p+c,q}(b;d)$ where $C\in \QQ$ is the appropriate constant so that this is a discrete probability distribution and $c\in\ZZ$ is such that the first non-zero value occurs when $p=0$. We then can compare these distributions to others in hopes of shedding light on the normality conjecture. 

One way to compare our data to a normal distribution is by creating a quantile-quantile (or Q-Q) plot. Specifically, having fixed $q,b,$ and $d$ we consider the Q-Q plot comparing the Betti distribution for $(q,b,d)$ to the normal distribution of best fit. If these two distributions were approximately the same we would expect the points in the Q-Q plot to be roughly distributed along the line $y=x$. Our resulting plots provide mild, but limited, evidence for the conjecture. 

Examining these Q-Q plots in further detail it seems that noise in the tails of the rows often muddies plots. In light of this -- and as the tails of the row are unlikely to effect any form of convergence to a normal distribution -- we also performed the above procedure after truncating the first few and last entries of each row. 
These plots appear in Figure~\ref{fig:QQ} and provide the first computational evidence for ~\cite[Conjecture~B]{eel-betti} for any variety of dimension $>1$. These graphics not only support~\cite[Conjecture~B]{eel-betti}, but they suggest that the normally distributed might kick even for modest values of $d$. 

\subsection{Boij-S\"oderberg coefficients}\label{subsec:BS coeffs}
Boij-S\"oderberg theory shows that the Betti table of any graded module can be decomposed as a positive rational sum of certain building blocks known as pure diagrams. The first proof of the main result appears in~\cite[Theorem~0.2]{eisenbud-schreyer-JAMS}, and \cite{floystad-expository} provides an expository treatment of the theory, including definitions of the relevant terms. As a consequence of Boij-S\"oderberg theory, we can study the rational coefficients that arise in this decomposition.

To get coefficients that are well-defined, we need to choose a basis for the pure diagrams $\pi_{\bd}$. Given a degree sequence $\bd=(d_0,d_1,\dots,d_r)$ we set $\pi_{\bd}$ as the Betti table (with rational entries) where
\[
\beta_{i,j}(\pi_{\bd}):= \begin{cases}
\prod_{i\ne j} \frac{1}{|d_i-d_j|} & \text{ if $j=d_i$}\\
0 & \text{ if $j\ne d_i$}.
\end{cases}
\]
For instance
\[
\pi_{(0,2,3)} = \begin{pmatrix}
\frac{1}{6} & - & -\\
-&\frac{1}{2}& \frac{1}{3}
\end{pmatrix}
\]

Then for any graded module $M$ over a polynomial ring, there exists a chain of degree sequences $C_M$ such that we can uniquely write
\[
\beta(M) = \sum_{\bd\in C_M} a_\bd\pi_\bd \text{ with } a_\bd\in \QQ_{>0}.
\]
See for instance~\cite[Theorem~5.1]{floystad-expository} for a discussion of these decompositions and their uniqueness properties.  We define the \defi{Boij-S\"oderberg coefficents of $M$} as the sequence $(a_\bd)_{\bd\in C_M}$.
\begin{example}
Let $S=\QQ[x,y,z]$ and $I=\langle x^2,xy,y^4\rangle$. Then we have the decomposition
\[
\beta(S/I) = \begin{pmatrix}1&-&-\\-&2&1\\-&-&-\\-&1&1 \end{pmatrix} =
3 \cdot \begin{pmatrix}\frac{1}{6}&-&-\\-&\frac{1}{2}&\frac{1}{3}\\-&-&-\\-&-&- \end{pmatrix} 
+3 \cdot \begin{pmatrix}\frac{1}{10}&-&-\\-&\frac{1}{6}&-\\-&-&-\\-&-&\frac{1}{15} \end{pmatrix} 
+4 \cdot \begin{pmatrix}\frac{1}{20}&-&-\\-&-&-\\-&-&-\\-&\frac{1}{4}&\frac{1}{5} \end{pmatrix} 
\]
and thus the Boij-S\"oderberg coefficients of $S/I$ are $(3,3,4)$.
\end{example}
In \cite[\S3]{eel-betti}, the Boij-S\"oderberg coefficients of Veronese varieties are shown to be closely connected to the conjectural ``normal distribution'' property discussed in the previous section, and thus Question~\ref{q:normal distribution} naturally raises the following question:
\begin{question}\label{q:BS coeffs}
For fixed $b$ and $d\to \infty$, how do the Boij-S\"oderberg coefficients behave? 
\end{question}

The limited data we have gathered suggests that the Boij-S\"oderberg coefficients are unlikely to be evenly or sporadically distributed; see Example~\ref{ex:BS coeffs} and Figure~\ref{fig:BS coeffs}. In fact, we conjecture:

\begin{conj}
For any $b,d$, the Boij-S\"oderberg coefficients of $\beta(\PP^2, b;d)$ are unimodal.
\end{conj}
\noindent We restrict this conjecture to $\PP^2$ because more complicated overlaps between rows will arise for $\PP^n$ with $n\geq 3$ and $d\gg 0$, and our data is insufficient to shed light on that.

One can sharpen Question~\ref{q:BS coeffs} in other ways: does one of the coefficients dominate, as in~\cite{erman-high-degree}? Under appropriate rescaling, will the coefficients converge to a reasonable function in the limit? 

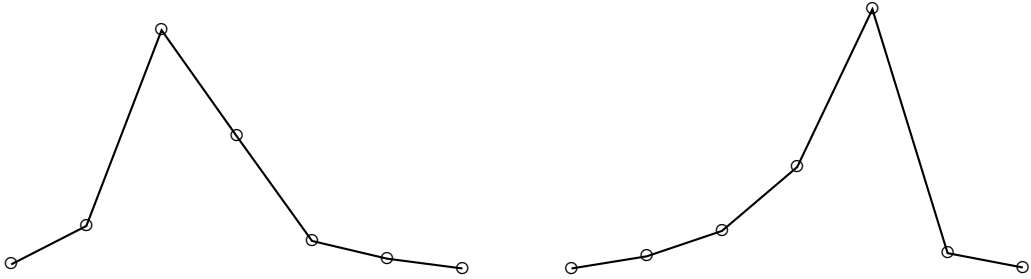
\begin{figure}
\begin{tikzpicture}[yscale=0.4]
\draw (0,.26) node{$\circ$};
\draw (1,1.54) node{$\circ$};
\draw (2,8.05) node{$\circ$};
\draw (3,4.52) node{$\circ$};
\draw (4,1.04) node{$\circ$};
\draw(5,.45) node{$\circ$};
\draw(6, .125) node{$\circ$};
\draw[-,thick] (0,.26)--(1,1.54)--(2,8.05)--(3,4.52)--(4,1.04)--(5,.45)--(6,.125);
\end{tikzpicture}
\qquad
\begin{tikzpicture}[yscale=1.1]
\draw (0,.060696) node{$\circ$};
\draw (1,.207654) node{$\circ$};
\draw (2,.514584) node{$\circ$};
\draw (3,1.28747) node{$\circ$};
\draw (4,3.1955) node{$\circ$};
\draw(5,.245905) node{$\circ$};
\draw(6, .0722298) node{$\circ$};
\draw[-,thick](0,.060696)-- (1,.207654)--(2,.514584)--(3,1.28747) --(4,3.1955) --(5,.245905)--(6, .0722298);
\end{tikzpicture}
\caption{We plot the Boij-S\"oderberg coefficients of: $\cO_{\PP^2}(3)$ under the embedding by $d=5$ (on the left) and $\cO_{\PP^2}$ under the embedding by $d=6$ (on the right). See also Example~\ref{ex:BS coeffs}.}
\label{fig:BS coeffs}
\end{figure}

\begin{example}\label{ex:BS coeffs}
We compute the Betti table $\beta(\PP^2,3;5)$ to be

{\fontsize{5}{6}
\[
\begin{array}{*{19}c}
0&1&2&3&4&5&6&7&8&9&10&11&12&13&14&15&16&17&18\\
10&165&1260&5865&18360&39900&58695&49419&12870&2002&\text
     {.}&\text{.}&\text{.}&\text{.}&\text{.}&\text{.}&\text{.}&\text{.}&\text{.}\\
\text{.}&\text{.}&\text{.}&\text{.}&120&1575&9639&52650&172172&291720&
     338130&291720&192780&97920&37740&10710&2115&260&15\\
\text{.}&\text{.}&\text{.}&\text{.}&\text{.}&\text{.}&\text{.}&\text{.}&\text{.}&\text{.}&\text{.}&\text{.}&\text{.}&\text{.}&\text{.}&\text{.}&\text{.}&\text{.}&\text
     {.}\\
     \end{array}
\]
}

\noindent The Boij-S\"oderberg coefficients are massive. For instance, the first coefficient is $2636271525888000$. To make the coefficients more reasonable, we rescale by $10^{-16}$ and round off, yielding the sequence of (rescaled) Boij-S\"oderberg coefficients
\[
(.263627 \ , \ 1.5441 \ , \ 8.05149 \ , \  4.52584 \ , \ 1.04027 \ , \ .455071 \ , \ .125537)
\]
Thes are plotted on the left in Figure~\ref{fig:BS coeffs}.
\end{example}

\subsection{Unimodality}\label{subsec:unimodal}
Many natural statistics associated to the syzygies of Veronese embeddings appear to always behave unimodally. This leads us to propose the following question.
\begin{question}\label{q:unimodal}
Fix $d,n,b$ and $q$. When is each of the following a unimodal function of $p$?
\begin{enumerate}[noitemsep]
	\item The rank of $K_{p,q}(\PP^n,b;d)$;
	\item The number of distinct irreducible Schur modules appearing in $K_{p,q}(\PP^n,b;d)$;
	\item The total number of irreducible Schur modules appearing in $K_{p,q}(\PP^n,b;d)$;
	\item The largest multiplicity of a Schur module in $K_{p,q}(\PP^n, b;d)$;
	\item\label{item:dom} The number of dominant weights in $K_{p,q}(\PP^n,b;d)$.
\end{enumerate}
\end{question}
The data suggests that these functions are nearly always unimodal.  This would not be surprising, especially in light of the conjectural normally distributed behavior of the Betti numbers. However, proving unimodality of one of the above functions might be a more tractable first step towards Question~\ref{q:normal distribution}.

\begin{remark}
In Question~\ref{q:unimodal}\eqref{item:dom} unimodality fails for $d=3$ and $b=0$. See Example~\ref{ex:lattice}.  This is the only known failure of unimodality that we are aware of.
\end{remark}

\begin{example}
On $\PP^2$, we consider the case $b=2$ and $d=4$ and $q=0$. See Appendix 1 for the Betti table. The rank of $K_{p,0}(\PP^2,2;4)$ is
$
(6, 62, 276, 660, 825, 252)
$ for $0\leq p \leq 5$, and the rank is $0$ for other values of $p$.
The number of irreducible Schur modules (with multiplicity) in $K_{p,0}(\PP^2,2;4)$ is
$
(1, 2, 7, 12, 13, 5)
$
for $0\leq p \leq 5$. 
\end{example}

\begin{example}
On $\PP^2$ we consider the case $b=2$ and $d=5$ and $q=0$.
The number of irreducible Schur modules (with multiplicity) for $K_{p,0}(\PP^2,2;5)$ is plotted in Figure~\ref{fig:unimodal}.
The number of dominant weights in $K_{p,0}(\PP^2,2;5)$ is $(1, 2, 2, 3, 3, 3, 4, 3, 3, 3, 3, 3, 2, 2, 1)$, and is also plotted in Figure~\ref{fig:unimodal}.

\begin{figure}
\begin{tikzpicture}[xscale=0.36, yscale=0.15]
\draw[->] (3,0) -- (17,0);
\draw[->] (3,0) -- (3,23);
\draw (2,0) node{\tiny $0$};
\draw (2,20) node{\tiny $2000$};
\draw (-0.5,12) node{\tiny\# Schur modules};
\draw (-0.5,10) node{\tiny in $K_{p,1}(2;5)$};
\draw (3,-2) node{\tiny $3$};
\draw (10,-2) node{\tiny $p$};
\draw (17,-2) node{\tiny $17$};
\draw[-,thick] (17, .01)--(16, .15)--(15, 1.67)--(14, 5.57)--(13, 11.67)--(12, 18.19)--(11, 22.28)--(10,21.94)--(9, 17.57)--(8, 11.41)--(7, 5.95)--(6, 2.46)--(5, .78)--(4, .17)--(3,.02);
\end{tikzpicture}
\quad
\begin{tikzpicture}[xscale=0.36, yscale=.75]
\draw[->] (3,0) -- (17,0);
\draw[->] (3,0) -- (3,5);
\draw (2,0) node{\tiny $0$};
\draw (2,4) node{\tiny $4$};
\draw (-0.5,2.4) node{\tiny\# dominant};
\draw (-0.5,2.0) node{\tiny weights in};
\draw (-0.5,1.6) node{\tiny $K_{p,1}(2;5)$};
\draw (3,-.5) node{\tiny $3$};
\draw (10,-.5) node{\tiny $p$};
\draw (17,-.5) node{\tiny $17$};
\draw[-,thick] (17, 1)--(16, 2)--(15, 2)--(14, 3)--(13, 3)--(12, 3)--(11, 4)--(10, 3)--(9, 3)--(8, 3)--(7, 3)--(6, 3)--(5, 2)--(4, 2)--(3, 1);
\end{tikzpicture}
\caption{The left plots the number of Schur modules (with multiplicity) in $K_{p,1}(\PP^2,2;5)$ for $3\leq p \leq 17$. The right plots the number of dominant weights for the same input. Both plots are unimodal, as suggested in Question~\ref{q:unimodal}.}
\end{figure}
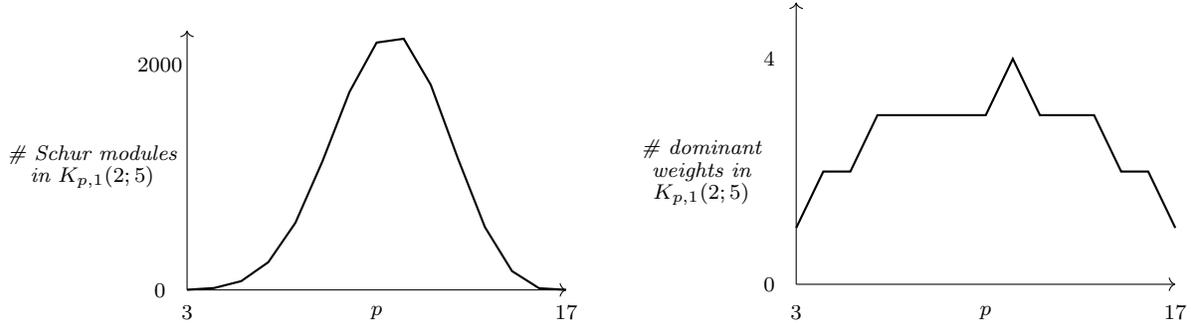
\label{fig:unimodal}
\end{example}

\subsection{Redundancy}\label{subsec:redundancy}
The following question was first posed to us, in various forms, by Eisenbud, Lazarsfeld, and Raicu. We focus on the case of $\PP^2$ for simplicity.
\begin{question}\label{q:redudancy}
Fix $b$ and let $d\gg 0$. Are most syzygies of $\cO_{\PP^2}(b)$ under the $d$-uple embedding determined by the Hilbert series? Or are most syzygies irrelevant to the Hilbert series? More precisely, for which $\epsilon>0$ and $d\gg 0$ can we find some $(p,q)$ where 
	\begin{equation*}\label{eqn:epsilon}
	\left|1 - \tfrac{\dim K_{p,q}(\PP^2,b;d)} {\dim K_{p-1,q+1}(\PP^2,b;d)}\right| < \epsilon?
	\end{equation*}
\end{question}
Our data does not provide a clear indication of what to expect for this question. For $d=4$, the entry with the highest proportion of ``redundant'' syzygies comes in the case $b=2$, where we have $K_{5,0}(2;4)=252$ and $K_{4,1}(2;4) = 450$ and
\[
\left|1 - \tfrac{\dim K_{5,0}(\PP^2,2;4)} {\dim K_{4,1}(\PP^2,2;4)}\right|=0.44.
\]
For $d=5$ and $d=6$, the most redundant entries also occur for $b=2$ and $(p,q)=(5,0)$. In the case $d=5$ we have $\left|1 - \tfrac{\dim K_{5,0}(\PP^2,2;5)} {\dim K_{5,1}(\PP^2,2;5)}\right|\approx 0.59$, and for $d=6$ the corresponding value is $\approx 0.57$. It would be interesting to better understand what is possible in the limit as $d\to \infty$.

One could also questions about redundancy of multigraded Betti numbers or of Schur functors. For instance, a folklore question had been to produce an example where redundant Schur modules appear. For $b=0$, we find that first case arises when $d=5$, where both $K_{14,1}(0;5)$ and $K_{13,2}(0;5)$ have a copy of $\bS_{(30,25,20)}$. 
\begin{question}
Do such redundant Schur modules appear frequently or only sporadically?
\end{question}
\begin{example}
For all $d\leq 5$ and all $0\leq b <4$, the only redundant Schur modules that arise are $\bS_{(30,25,20)}$ which arises in both $K_{14,1}(0;5)$ and $K_{13,2}(0;5)$; $\bS_{(30,24,21)}$ which arises in both $K_{14,1}(0;5)$ and $K_{13,2}(0;5)$; and the dual examples for $b=3$.
\end{example}


%
%
%
%
%

\newpage

\begin{landscape}

\section*{Appendix 1: Total Betti Numbers}

Here we include the Betti tables for $n=2$, $d=4,5$, and $0\leq b\leq d$.  Additional data such as multigraded Betti numbers, and further examples is available at \href{https://syzygydata.com}{syzygydata.com}.

{\fontsize{11}{13.2}
\[
\beta(2,0;4)=
\begin{array}{*{19}c}
0&1&2&3&4&5&6&7&8&9&10&11&12\\
1&\text{.}&\text{.}&\text{.}&\text{.}&\text{.}&\text{.}&\text{.}&\text{.}&\text{.}&\text{.}&\text{.}&\text{.}\\
\text{.}&75&536&1947&4488&7095&7920&6237&3344&1089&120&\text{.}&\text{.}\\
\text{.}&\text{.}&\text{.}&\text{.}&\text{.}&\text{.}&\text{.}&\text{.}&\text{.}&\text{.}&55&24&3
\end{array}
\]
}

{\fontsize{11}{13.2}
\[
\beta(2,1;4)=
\begin{array}{*{19}c}
0&1&2&3&4&5&6&7&8&9&10&11&12\\
3&24&55&\text{.}&\text{.}&\text{.}&\text{.}&\text{.}&\text{.}&\text{.}&\text{.}&\text{.}&\text{.}\\
\text{.}&\text{.}&120&1089&3344&6237&7920&7095&4488&1947&536&75&\text{.}\\
\text{.}&\text{.}&\text{.}&\text{.}&\text{.}&\text{.}&\text{.}&\text{.}&\text{.}&\text{.}&\text{.}&\text{.}&1
\end{array}
\]
}

{\fontsize{11}{13.2}
\[
\beta(2,2;4)=
\begin{array}{*{19}c}
0&1&2&3&4&5&6&7&8&9&10&11&12\\
6&62&276&660&825&252&\text{.}&\text{.}&\text{.}&\text{.}&\text{.}&\text{.}&\text{.}\\
\text{.}&\text{.}&\text{.}&55&450&2376&4488&4950&3630&1804&588&114&10\\
\text{.}&\text{.}&\text{.}&\text{.}&\text{.}&\text{.}&\text{.}&\text{.}&\text{.}&\text{.}&\text{.}&\text{.}&\text{.}
\end{array}
\]
}

{\fontsize{11}{13.2}
\[
\beta(2,3;4)=
\begin{array}{*{19}c}
0&1&2&3&4&5&6&7&8&9&10&11&12\\
10&114&588&1804&3630&4950&4488&2376&450&55&\text{.}&\text{.}&\text{.}\\
\text{.}&\text{.}&\text{.}&\text{.}&\text{.}&\text{.}&\text{.}&252&825&660&276&62&6\\
\text{.}&\text{.}&\text{.}&\text{.}&\text{.}&\text{.}&\text{.}&\text{.}&\text{.}&\text{.}&\text{.}&\text{.}&\text{.}
\end{array}
\]
}
\end{landscape}

\newpage

\begin{landscape}
{\fontsize{9}{10.8}
\[
\beta(2,0;5)=
\begin{array}{*{19}c}
0&1&2&3&4&5&6&7&8&9&10&11&12&13&14&15&16&17&18\\
1&\text{.}&\text{.}&\text{.}&\text{.}&\text{.}&\text{.}&\text{.}&\text{.}&\text{.}&\text{.}&\text{.}&\text{.}&\text{.}&\text{.}&\text{.}&\text{.}&\text{.}&\text{.}\\
\text{.}&165&1830&10710&41616&117300&250920&
     417690&548080&568854&464100&291720&134640&39780&4858&375&\text{.}&\text{.}&\text{.}\\
\text{.}&\text{.}&\text{.}&\text{.}&\text{.}&\text{.}&\text{.}&\text{.}&\text{.}&\text{.}&\text{.}&\text{.}&\text{.}&2002&4200&2160&595&90&6\\ 
\end{array}
\]
}

{\fontsize{9}{10.8}
\[
\beta(2,1;5)=
\begin{array}{*{19}c}
0&1&2&3&4&5&6&7&8&9&10&11&12&13&14&15&16&17&18\\
3&35&120&\text{.}&\text{.}&\text{.}&\text{.}&\text{.}&\text{.}&\text{.}&\text{.}&\text{.}&\text{.}&\text{.}&\text{.}&\text{.}&\text{.}&\text{.}&\text{.}\\
\text{.}&\text{.}&405&5865&29988&97920&231540&417690 &590070&661232&590070&417690&231540&97920&29988&5865&405&\text{.}&\text{.}\\
\text{.}&\text{.}&\text{.}&\text{.}&\text{.}&\text{.}&\text{.}&\text{.}&\text{.}&\text{.}&\text{.}&\text{.}&\text{.}&\text{.}&\text{.}&\text{.}&120&35&3\\
\end{array}
\]
}

{\fontsize{9}{10.8}
\[
\beta(2,2;5)=
\begin{array}{*{19}c}
0&1&2&3&4&5&6&7&8&9&10&11&12&13&14&15&16&17&18\\
   6&90&595&2160&4200&2002&\text{.}&\text{.}&\text{.}&\text{.}&\text{.}&\text{.}&\text{.}&\text{.}&\text{.}&\text{.}&\text{.}&\text{.}&\text{.}\\
\text{.}&\text{.}&\text{.}&375&4858&39780&134640&291720&464100&568854&548080&417690&250920&117300&41616&10710&1830&165&\text{.}\\
\text{.}&\text{.}&\text{.}&\text{.}&\text{.}&\text{.}&\text{.}&\text{.}&\text{.}&\text{.}&\text{.}&\text{.}&\text{.}&\text{.}&\text{.}&\text{.}&\text{.}&\text{.}&1\\
\end{array}
\]
}

{\fontsize{9}{10.8}
\[
\beta(2,3;5)=
\begin{array}{*{19}c}
0&1&2&3&4&5&6&7&8&9&10&11&12&13&14&15&16&17&18\\
10&165&1260&5865&18360&39900&58695&49419&12870&2002&\text{.}&\text{.}&\text{.}&\text{.}&\text{.}&\text{.}&\text{.}&\text{.}&\text{.}\\   
\text{.}&\text{.}&\text{.}&\text{.}&120&1575&9639&52650&172172&291720&338130&291720&192780&97920&37740&10710&2115&260&15\\
\text{.}&\text{.}&\text{.}&\text{.}&\text{.}&\text{.}&\text{.}&\text{.}&\text{.}&\text{.}&\text{.}&\text{.}&\text{.}&\text{.}&\text{.}&\text{.}&\text{.}&\text{.}&\text{.}\\
     \end{array}
     \]
     }
     
{\fontsize{9}{10.8}
\[
\beta(2,4;5)=
\begin{array}{*{19}c}
0&1&2&3&4&5&6&7&8&9&10&11&12&13&14&15&16&17&18\\
15&260&2115&10710&37740&97920&192780&291720&338130&291720
      &172172&52650&9639&1575&120&\text{.}&\text{.}&\text{.}&\text{.}\\
      
      \text{.}&\text{.}&\text{.}&\text{.}&\text{.}&\text{.}&\text{.}&\text{.}&\text{.}&2002
      &12870&49419&58695&39900&18360&5865&1260&165&10\\
      
      \text{.}&\text{.}&\text{.}&\text{.}&\text{.}&\text{.}&\text{.}&\text{.}&\text{.}&\text{.}&\text
      {.}&\text{.}&\text{.}&\text{.}&\text{.}&\text{.}&\text{.}&\text{.}&\text{.}\\
\end{array}
\]
}

\end{landscape}

\newpage
\section*{Appendix 2: Schur Module Decompositions}\label{appendix:schur module}

Here we include the Schur module decomposition of $K_{p,q}(5,0)$ and $K_{p,q}(5,3)$ for all $(p,q)$ in the relevant range. The dominant Schur modules are colored in red. Complete Schur module decompositions for the remainder of the computed examples is available at \href{https://syzygydata.com}{syzygydata.com}.

{\fontsize{8}{9.6}
\begin{align*}
K_{14,1}(5;0)\cong{\color{red}\bS_{(34,21,20)}}&\oplus{\color{red}\bS_{(33,25,17)}}\oplus\bS_{(33,24,18)}\oplus\bS_{(33,23,19)}\oplus\bS_{(33,22,20)}\oplus\bS_{(32,25,18)}\oplus\bS_{(32,24,19)}\oplus\bS_{(32,23,20)}\oplus\bS_{(32,22,21)}\\
&\oplus\bS_{(31,25,19)}\oplus\bS_{(31,24,20)}\oplus\bS_{(31,23,21)}\oplus\bS_{(30,25,20)}\oplus\bS_{(30,24,21)}\oplus\bS_{(29,25,21)}
\end{align*}
\begin{align*}
K_{15,1}(5;0)\cong{\color{red}\bS_{(34,25,21)}}
\end{align*}
\begin{align*}
K_{13,2}(5;0)\cong{\color{red}\bS_{(30,30,15)}}&\oplus\bS_{(30,28,17)}\oplus\bS_{(30,27,18)}\oplus\bS_{(30,26,19)}\oplus\bS_{(30,25,20)}\oplus\bS_{(30,24,21)}\oplus\bS_{(29,26,20)}\oplus\bS_{(29,24,22)}\oplus\bS_{(28,28,19)}\\
&\oplus\bS_{(28,27,20)}\oplus\bS_{(28,26,21)}^{2}\oplus\bS_{(28,25,22)}\oplus\bS_{(28,24,23)}\oplus\bS_{(27,26,22)}\oplus\bS_{(27,24,24)}\oplus\bS_{(26,26,23)}
\end{align*}
\begin{align*}
K_{14,2}(5;0)\cong{\color{red}\bS_{(32,30,18)}}&\oplus\bS_{(32,28,20)}^{2}\oplus\bS_{(32,27,21)}\oplus\bS_{(32,26,22)}^{2}\oplus\bS_{(32,24,24)}^{2}\oplus\bS_{(31,30,19)}\oplus\bS_{(31,29,20)}\oplus\bS_{(31,28,21)}^{2}\oplus\bS_{(31,27,22)}^{2}\\
&\oplus\bS_{(31,26,23)}^{2}\oplus\bS_{(31,25,24)}\oplus\bS_{(30,30,20)}\oplus\bS_{(30,29,21)}\oplus\bS_{(30,28,22)}^{3}\oplus\bS_{(30,27,23)}^{2}\oplus\bS_{(30,26,24)}^{3}\oplus\bS_{(29,28,23)}^{2}\\
&\oplus\bS_{(29,27,24)}^{2}\oplus\bS_{(29,26,25)}\oplus\bS_{(28,28,24)}^{2}\oplus\bS_{(28,27,25)}\oplus\bS_{(28,26,26)}
\end{align*}
\begin{align*}
K_{5,0}(5;3)\cong{\color{red}\bS_{(20,4,4)}}&\oplus{\color{red}\bS_{(19,7,2)}}\oplus\bS_{(19,6,3)}\oplus\bS_{(19,5,4)}\oplus\bS_{(18,8,2)}\oplus\bS_{(18,7,3)}^{2}\oplus\bS_{(18,6,4)}^{3}\oplus\bS_{(17,9,2)}^{2}\oplus\bS_{(17,8,3)}^{4}\\
&\oplus\bS_{(17,7,4)}^{6}\oplus\bS_{(17,6,5)}^{3}\oplus\bS_{(16,10,2)}^{2}\oplus\bS_{(16,9,3)}^{5}\oplus\bS_{(16,8,4)}^{9}\oplus\bS_{(16,7,5)}^{6}\oplus\bS_{(16,6,6)}^{4}\oplus\bS_{(15,11,2)}^{3}\\
&\oplus\bS_{(15,10,3)}^{6}\oplus\bS_{(15,9,4)}^{10}\oplus\bS_{(15,8,5)}^{11}\oplus\bS_{(15,7,6)}^{8}\oplus\bS_{(14,12,2)}\oplus\bS_{(14,11,3)}^{5}\oplus\bS_{(14,10,4)}^{11}\oplus\bS_{(14,9,5)}^{12}\\
&\oplus\bS_{(14,8,6)}^{13}\oplus\bS_{(14,7,7)}^{4}\oplus\bS_{(13,13,2)}\oplus\bS_{(13,12,3)}^{3}\oplus\bS_{(13,11,4)}^{8}\oplus\bS_{(13,10,5)}^{12}\oplus\bS_{(13,9,6)}^{13}\oplus\bS_{(13,8,7)}^{9}\\
&\oplus\bS_{(12,12,4)}^{3}\oplus\bS_{(12,11,5)}^{7}\oplus\bS_{(12,10,6)}^{12}\oplus\bS_{(12,9,7)}^{9}\oplus\bS_{(12,8,8)}^{5}\oplus\bS_{(11,11,6)}^{4}\oplus\bS_{(11,10,7)}^{7}\oplus\bS_{(11,9,8)}^{4}\\
&\oplus\bS_{(10,10,8)}^{3}\oplus\bS_{(10,9,9)}
\end{align*}
\begin{align*}
K_{6,0}(5;3)\cong{\color{red}\bS_{(22,7,4)}}&\oplus{\color{red}\bS_{(21,9,3)}}\oplus\bS_{(21,8,4)}\oplus\bS_{(21,7,5)}^{2}\oplus\bS_{(20,9,4)}^{3}\oplus\bS_{(20,8,5)}^{3}\oplus\bS_{(20,7,6)}^{2}\oplus\bS_{(19,11,3)}^{2}\oplus\bS_{(19,10,4)}^{4}\\
&\oplus\bS_{(19,9,5)}^{7}\oplus\bS_{(19,8,6)}^{5}\oplus\bS_{(19,7,7)}^{4}\oplus\bS_{(18,12,3)}\oplus\bS_{(18,11,4)}^{5}\oplus\bS_{(18,10,5)}^{7}\oplus\bS_{(18,9,6)}^{10}\oplus\bS_{(18,8,7)}^{6}\\
&\oplus\bS_{(17,13,3)}^{2}\oplus\bS_{(17,12,4)}^{5}\oplus\bS_{(17,11,5)}^{12}\oplus\bS_{(17,10,6)}^{12}\oplus\bS_{(17,9,7)}^{14}\oplus\bS_{(17,8,8)}^{4}\oplus\bS_{(16,13,4)}^{5}\oplus\bS_{(16,12,5)}^{8}\\
&\oplus\bS_{(16,11,6)}^{15}\oplus\bS_{(16,10,7)}^{15}\oplus\bS_{(16,9,8)}^{10}\oplus\bS_{(15,15,3)}^{2}\oplus\bS_{(15,14,4)}^{2}\oplus\bS_{(15,13,5)}^{9}\oplus\bS_{(15,12,6)}^{13}\oplus\bS_{(15,11,7)}^{19}\\
&\oplus\bS_{(15,10,8)}^{13}\oplus\bS_{(15,9,9)}^{8}\oplus\bS_{(14,14,5)}\oplus\bS_{(14,13,6)}^{7}\oplus\bS_{(14,12,7)}^{11}\oplus\bS_{(14,11,8)}^{15}\oplus\bS_{(14,10,9)}^{8}\oplus\bS_{(13,13,7)}^{8}\\
&\oplus\bS_{(13,12,8)}^{8}\oplus\bS_{(13,11,9)}^{11}\oplus\bS_{(13,10,10)}^{3}\oplus\bS_{(12,12,9)}^{3}\oplus\bS_{(12,11,10)}^{3}\oplus\bS_{(11,11,11)}^{2}
\end{align*}
\begin{align*}
K_{7,0}(5;3)\cong{\color{red}\bS_{(24,9,5)}}&\oplus\bS_{(24,7,7)}\oplus\bS_{(23,10,5)}\oplus\bS_{(23,9,6)}\oplus\bS_{(23,8,7)}\oplus\bS_{(22,11,5)}^{2}\oplus\bS_{(22,10,6)}^{2}\oplus\bS_{(22,9,7)}^{4}\oplus\bS_{(21,12,5)}^{2}\\
&\oplus\bS_{(21,11,6)}^{3}\oplus\bS_{(21,10,7)}^{5}\oplus\bS_{(21,9,8)}^{4}\oplus\bS_{(20,13,5)}^{3}\oplus\bS_{(20,12,6)}^{4}\oplus\bS_{(20,11,7)}^{8}\oplus\bS_{(20,10,8)}^{5}\oplus\bS_{(20,9,9)}^{5}\\
&\oplus\bS_{(19,14,5)}^{2}\oplus\bS_{(19,13,6)}^{5}\oplus\bS_{(19,12,7)}^{9}\oplus\bS_{(19,11,8)}^{10}\oplus\bS_{(19,10,9)}^{7}\oplus\bS_{(18,15,5)}^{2}\oplus\bS_{(18,14,6)}^{4}\oplus\bS_{(18,13,7)}^{10}\\
&\oplus\bS_{(18,12,8)}^{10}\oplus\bS_{(18,11,9)}^{12}\oplus\bS_{(18,10,10)}^{2}\oplus\bS_{(17,16,5)}\oplus\bS_{(17,15,6)}^{3}\oplus\bS_{(17,14,7)}^{8}\oplus\bS_{(17,13,8)}^{12}\oplus\bS_{(17,12,9)}^{13}\\
&\oplus\bS_{(17,11,10)}^{8}\oplus\bS_{(16,16,6)}\oplus\bS_{(16,15,7)}^{5}\oplus\bS_{(16,14,8)}^{7}\oplus\bS_{(16,13,9)}^{13}\oplus\bS_{(16,12,10)}^{8}\oplus\bS_{(16,11,11)}^{6}\oplus\bS_{(15,15,8)}^{4}\\
&\oplus\bS_{(15,14,9)}^{7}\oplus\bS_{(15,13,10)}^{9}\oplus\bS_{(15,12,11)}^{6}\oplus\bS_{(14,14,10)}^{2}\oplus\bS_{(14,13,11)}^{5}\oplus\bS_{(14,12,12)}^{6}\oplus\bS_{(13,13,12)}^{20}
\end{align*}
\begin{align*}
K_{8,0}(5;3)\cong{\color{red}\bS_{(26,10,7)}}&\oplus\bS_{(25,10,8)}\oplus\bS_{(24,12,7)}\oplus\bS_{(24,11,8)}\oplus\bS_{(24,10,9)}^{2}\oplus\bS_{(23,12,8)}\oplus\bS_{(23,11,9)}\oplus\bS_{(23,10,10)}^{2}\oplus\bS_{(22,14,7)}\\
&\oplus\bS_{(22,13,8)}\oplus\bS_{(22,12,9)}^{3}\oplus\bS_{(22,11,10)}^{2}\oplus\bS_{(21,14,8)}\oplus\bS_{(21,13,9)}\oplus\bS_{(21,12,10)}^{3}\oplus\bS_{(20,16,7)}\oplus\bS_{(20,15,8)}\\
&\oplus\bS_{(20,14,9)}^{3}\oplus\bS_{(20,13,10)}^{3}\oplus\bS_{(20,12,11)}^{2}\oplus\bS_{(19,16,8)}\oplus\bS_{(19,15,9)}\oplus\bS_{(19,14,10)}^{3}\oplus\bS_{(19,13,11)}\oplus\bS_{(19,12,12)}^{2}\\
&\oplus\bS_{(18,18,7)}\oplus\bS_{(18,17,8)}\oplus\bS_{(18,16,9)}^{3}\oplus\bS_{(18,15,10)}^{3}\oplus\bS_{(18,14,11)}^{3}\oplus\bS_{(18,13,12)}^{2}\oplus\bS_{(17,16,10)}^{2}\oplus\bS_{(17,15,11)}\\
&\oplus\bS_{(17,14,12)}^{3}\oplus\bS_{(16,16,11)}^{2}\oplus\bS_{(16,15,12)}^{2}\oplus\bS_{(16,14,13)}^{2}\oplus\bS_{(15,14,14)}
\end{align*}
\begin{align*}
K_{9,0}(5;3)\cong{\color{red}\bS_{(28,10,10)}}&\oplus\bS_{(26,12,10)}\oplus\bS_{(24,14,10)}\oplus\bS_{(24,12,12)}\oplus\bS_{(22,16,10)}\oplus\bS_{(22,14,12)}\oplus\bS_{(20,18,10)}\oplus\bS_{(20,16,12)}\oplus\bS_{(20,14,14)}\\
&\oplus\bS_{(18,18,12)}\oplus\bS_{(18,16,14)}\oplus\bS_{(16,16,16)}
\end{align*}
\begin{align*}
K_{4,1}(5;3)\cong{\color{red}\bS_{(14,14,0)}}
\end{align*}
\begin{align*}
K_{5,1}(5;3)\cong{\color{red}\bS_{(18,14,1)}}\oplus\bS_{(17,14,2)}\oplus\bS_{(16,14,3)}\oplus\bS_{(15,14,4)}\oplus\bS_{(14,14,5)}
\end{align*}
\begin{align*}
K_{6,1}(5;3)\cong{\color{red}\bS_{(21,15,2)}}&\oplus\bS_{(21,14,3)}\oplus\bS_{(20,15,3)}\oplus\bS_{(20,14,4)}\oplus\bS_{(19,17,2)}\oplus\bS_{(19,16,3)}\oplus\bS_{(19,15,4)}^{2}\oplus\bS_{(19,14,5)}^{2}\oplus\bS_{(18,17,3)}\\
&\oplus\bS_{(18,16,4)}\oplus\bS_{(18,15,5)}^{2}\oplus\bS_{(18,14,6)}^{2}\oplus\bS_{(17,17,4)}\oplus\bS_{(17,16,5)}\oplus\bS_{(17,15,6)}^{2}\oplus\bS_{(17,14,7)}^{2}\oplus\bS_{(16,15,7)}\\
&\oplus\bS_{(16,14,8)}\oplus\bS_{(15,15,8)}\oplus\bS_{(15,14,9)}\oplus\bS_{(14,12,12)}^{5}\oplus\bS_{(13,13,12)}^{18}
\end{align*}
\begin{align*}
K_{7,1}(5;3)\cong{\color{red}\bS_{(24,15,4)}}&\oplus{\color{red}\bS_{(23,17,3)}}\oplus\bS_{(23,16,4)}\oplus\bS_{(23,15,5)}^{3}\oplus\bS_{(23,14,6)}\oplus\bS_{(23,13,7)}\oplus\bS_{(23,11,9)}\oplus\bS_{(22,17,4)}^{2}\oplus\bS_{(22,16,5)}^{2}\\
&\oplus\bS_{(22,15,6)}^{4}\oplus\bS_{(22,14,7)}^{2}\oplus\bS_{(22,13,8)}\oplus\bS_{(22,12,9)}\oplus\bS_{(22,11,10)}\oplus\bS_{(21,19,3)}\oplus\bS_{(21,18,4)}^{2}\oplus\bS_{(21,17,5)}^{5}\\
&\oplus\bS_{(21,16,6)}^{5}\oplus\bS_{(21,15,7)}^{8}\oplus\bS_{(21,14,8)}^{3}\oplus\bS_{(21,13,9)}^{3}\oplus\bS_{(21,12,10)}\oplus\bS_{(21,11,11)}^{2}\oplus\bS_{(20,19,4)}\oplus\bS_{(20,18,5)}^{2}\\
&\oplus\bS_{(20,17,6)}^{6}\oplus\bS_{(20,16,7)}^{5}\oplus\bS_{(20,15,8)}^{8}\oplus\bS_{(20,14,9)}^{4}\oplus\bS_{(20,13,10)}^{3}\oplus\bS_{(20,12,11)}^{2}\oplus\bS_{(19,19,5)}^{3}\oplus\bS_{(19,18,6)}^{4}\\
&\oplus\bS_{(19,17,7)}^{9}\oplus\bS_{(19,16,8)}^{8}\oplus\bS_{(19,15,9)}^{10}\oplus\bS_{(19,14,10)}^{4}\oplus\bS_{(19,13,11)}^{4}\oplus\bS_{(18,18,7)}\oplus\bS_{(18,17,8)}^{6}\oplus\bS_{(18,16,9)}^{5}\\
&\oplus\bS_{(18,15,10)}^{8}\oplus\bS_{(18,14,11)}^{3}\oplus\bS_{(18,13,12)}^{2}\oplus\bS_{(17,17,9)}^{6}\oplus\bS_{(17,16,10)}^{4}\oplus\bS_{(17,15,11)}^{7}\oplus\bS_{(17,14,12)}^{2}\oplus\bS_{(17,13,13)}^{2}\\
&\oplus\bS_{(16,15,12)}^{3}\oplus\bS_{(16,14,13)}\oplus\bS_{(15,15,13)}^{3}
\end{align*}
\begin{align*}
K_{8,1}(5;3)\cong{\color{red}\bS_{(26,17,5)}}&\oplus\bS_{(26,16,6)}\oplus\bS_{(26,15,7)}^{2}\oplus\bS_{(26,13,9)}^{2}\oplus\bS_{(26,11,11)}\oplus\bS_{(25,18,5)}\oplus\bS_{(25,17,6)}^{2}\oplus\bS_{(25,16,7)}^{3}\oplus\bS_{(25,15,8)}^{4}\\
&\oplus\bS_{(25,14,9)}^{3}\oplus\bS_{(25,13,10)}^{2}\oplus\bS_{(25,12,11)}^{2}\oplus{\color{red}\bS_{(24,20,4)}}\oplus\bS_{(24,19,5)}^{2}\oplus\bS_{(24,18,6)}^{4}\oplus\bS_{(24,17,7)}^{8}\oplus\bS_{(24,16,8)}^{8}\\
&\oplus\bS_{(24,15,9)}^{10}\oplus\bS_{(24,14,10)}^{6}\oplus\bS_{(24,13,11)}^{7}\oplus\bS_{(24,12,12)}\oplus\bS_{(23,20,5)}^{2}\oplus\bS_{(23,19,6)}^{5}\oplus\bS_{(23,18,7)}^{9}\oplus\bS_{(23,17,8)}^{13}\\
&\oplus\bS_{(23,16,9)}^{16}\oplus\bS_{(23,15,10)}^{15}\oplus\bS_{(23,14,11)}^{11}\oplus\bS_{(23,13,12)}^{7}\oplus\bS_{(22,21,5)}\oplus\bS_{(22,20,6)}^{4}\oplus\bS_{(22,19,7)}^{9}\oplus\bS_{(22,18,8)}^{15}\\
&\oplus\bS_{(22,17,9)}^{22}\oplus\bS_{(22,16,10)}^{21}\oplus\bS_{(22,15,11)}^{23}\oplus\bS_{(22,14,12)}^{12}\oplus\bS_{(22,13,13)}^{7}\oplus\bS_{(21,21,6)}\oplus\bS_{(21,20,7)}^{6}\oplus\bS_{(21,19,8)}^{13}\\
&\oplus\bS_{(21,18,9)}^{20}\oplus\bS_{(21,17,10)}^{27}\oplus\bS_{(21,16,11)}^{27}\oplus\bS_{(21,15,12)}^{22}\oplus\bS_{(21,14,13)}^{12}\oplus\bS_{(20,20,8)}^{5}\oplus\bS_{(20,19,9)}^{14}\oplus\bS_{(20,18,10)}^{21}\\
&\oplus\bS_{(20,17,11)}^{28}\oplus\bS_{(20,16,12)}^{26}\oplus\bS_{(20,15,13)}^{21}\oplus\bS_{(20,14,14)}^{4}\oplus\bS_{(19,19,10)}^{9}\oplus\bS_{(19,18,11)}^{17}\oplus\bS_{(19,17,12)}^{22}\oplus\bS_{(19,16,13)}^{19}\\
&\oplus\bS_{(19,15,14)}^{13}\oplus\bS_{(18,18,12)}^{7}\oplus\bS_{(18,17,13)}^{15}\oplus\bS_{(18,16,14)}^{10}\oplus\bS_{(18,15,15)}^{6}\oplus\bS_{(17,17,14)}^{5}\oplus\bS_{(17,16,15)}^{5}
\end{align*}

}
\end{document}